\documentclass{amsart}
\usepackage{amsfonts,amscd,amsthm,amsgen,amsmath,amssymb}
\usepackage[all]{xy}
\usepackage{epsfig,color}
\usepackage[vcentermath]{youngtab}
\newtheorem{theorem}{Theorem}[section]
\newtheorem{lemma}[theorem]{Lemma}
\newtheorem{corollary}[theorem]{Corollary}
\newtheorem{proposition}[theorem]{Proposition}

\theoremstyle{definition}
\newtheorem{definition}[theorem]{Definition}
\newtheorem{example}[theorem]{Example}

\theoremstyle{remark}

\numberwithin{equation}{section}

\begin{document}

\title[Knot concordance invariants from Seiberg--Witten theory]{Knot concordance invariants from Seiberg--Witten theory and slice genus bounds in $4$-manifolds}
\author{David Baraglia}

\address{School of Mathematical Sciences, The University of Adelaide, Adelaide SA 5005, Australia}

\email{david.baraglia@adelaide.edu.au}


\date{\today}

\begin{abstract}

We construct a new family of knot concordance invariants $\theta^{(q)}(K)$, where $q$ is a prime number. Our invariants are obtained from the equivariant Seiberg--Witten--Floer cohomology, constructed by the author and Hekmati, applied to the degree $q$ cyclic cover of $S^3$ branched over $K$. In the case $q=2$, our invariant $\theta^{(2)}(K)$ shares many similarities with the knot Floer homology invariant $\nu^+(K)$ defined by Hom and Wu. Our invariants $\theta^{(q)}(K)$ give lower bounds on the genus of any smooth, properly embedded, homologically trivial surface bounding $K$ in a definite $4$-manifold with boundary $S^3$.

\end{abstract}

\maketitle


\section{Introduction}

In this paper we introduce a series of new knot concordance invariants. Our invariants are obtained from the equivariant Seiberg--Witten--Floer cohomology, constructed by the author and Hekmati \cite{bh}, applied to cyclic branched covers of knots. Our invariants can be used to bound the slice genus of knots, or more generally the genus of any smooth, properly embedded, homologically trivial surface in a definite $4$-manifold with $S^3$ boundary which bounds the knot.

For each prime $q$, we obtain a knot concordance invariant $\theta^{(q)}(K)$. For simplicity we will restrict attention to the $q=2$ invariant in the introduction and write $\theta(K) = \theta^{(2)}(K)$.

\begin{theorem}\label{thm:theta1}
The invariant $\theta(K)$ is a knot concordance invariant valued in the non-negative integers. Moreover $\theta$ satisfies the following properties:
\begin{itemize}
\item[(1)]{$-\sigma(K)/2 \le \theta(K) \le g_4(K)$.}
\item[(2)]{$\theta(K_1 + K_2) \le \theta(K_1) + \theta(K_2)$.}
\item[(3)]{Let $K_+,K_-$ be knots where $K_-$ is obtained from $K_+$ by changing a positive crossing into a negative crossing. Then
\[
0 \le \theta(K_+) - \theta(K_-) \le 1.
\]
}
\item[(4)]{If $K$ is quasi-alternating, then 
\[
\theta(K) = \begin{cases} -\sigma(K)/2, & \sigma(K) \le 0, \\ 0, & \sigma(K) > 0. \end{cases}
\]
}
\item[(5)]{If $\delta(K) < -\sigma(K)/2$ and $\sigma(K) \le 0$, then $\theta(K) \ge 1 +\sigma(K)/2$. Here $\delta(K)$ is the Manolescu--Owens knot concordance invariant \cite{mo}.
}
\end{itemize}
\end{theorem}

Note that property (1) above means that the slice genus bound $g_4(K) \ge \theta(K)$ can be thought of as a refinement of the bound $g_4(K) \ge -\sigma(K)/2$. A similar situation occurs for the Ozsv\'ath--Szab\'o tau-invariant $\tau(K)$, namely there exists a non-negative, integer-valued concordance invariant $\nu^+(K)$ which refines the slice genus bound $g_4(K) \ge \tau(K)$ in the sense that $\tau(K) \le \nu^+(K) \le g_4(K)$ \cite{howu}. There are multiple similarities between $\theta$ and $\nu^+$. From \cite{howu} and \cite{bcg}, we have that $\nu^+(K)$ satisfies properties (2)-(5) of Theorem \ref{thm:theta1}. Another similarity is that $\nu^+(K)$ arises as the smallest $j$ such that $V_j(K) = 0$, where $V_j(K)$ is a sequence of decreasing knot concordance invariants originally defined in \cite{ras} (in a different notation), while $\theta(K)$ is similarly defined as the smallest $j$ for which a certain sequence of concordance invariants vanish.

For prime knots with $9$ or fewer crossings we have found that in all but one case we have $\theta(K) = \nu^+(K)$ and $\theta(-K) = \nu^+(-K)$. The exception is the knot $K = 9_{42}$ which has $\theta(K) = \nu^+(K) = 0$, but $\theta(-K) = 1$ and $\nu^+(-K) = 0$. It is interesting to note that $9_{42}$ is the only prime knot with $9$ or fewer crossings which is neither quasialternating nor quasipositive (or the mirror of a quasipositive knot).

The invariant $\theta(K)$ sometimes gives a better slice genus bound than can be obtained from the $\tau, \nu^+$ or $s$ invariants, or from the Levine--Tristram signature function.

\begin{example}\label{ex:bound}
Let $K = -9_{42} \# Wh(T_{2,2a+1})$, where $Wh(T_{2,2a+1})$ is the untwisted Whitehead double of the torus knot $T_{2,2a+1}$ with positive clasp and $a \ge 1$. Then $\tau(K) = 1$, $s(K)/2 = 1$, where $s$ is the Rasmussen $s$-invariant and $|\sigma_K(\omega)/2 | \le 1$ for all $\omega$, where $\sigma_K(\omega)$ is the Levine--Tristram signature function. We also have $\nu^+(K) \le 1$ and $\nu^+(-K) = 0$. Thus all these invariants give a slice genus bound $g_4(K) \ge 1$. However we will see that $\theta(K) = 2$, giving the better bound $g_4(K) \ge 2$. In fact, since $9_{42}$ and $Wh(T_{2,2a+1})$ both have slice genus $1$ and unknotting number $1$ we get that $g_4(K) = u(K) = 2$.
\end{example}

While it is not always easy to compute $\theta(K)$, we can obtain a lower bound for it in some circumstances. For instance, we have the following:

\begin{theorem}
Let $\Sigma_2(K)$ denote the double cover of $S^3$ branched over $K$ and let $\mathfrak{s}_0$ denote the unique spin$^c$-structure which arises from a spin structure. Let $\ell(K)$ denote the lowest $i$ for which $HF^+_i( \Sigma_2(K) , \mathfrak{s}_0 )$ is non-zero (where $HF^+$ is taken with coefficients in $\mathbb{F} = \mathbb{Z}/2\mathbb{Z}$). Then
\[
\theta(K) \ge \ell( -K ) - \frac{3 \sigma(K)}{4}.
\]
\end{theorem}

In particular this implies the following bound on the slice genus
\[
g_4(K) \ge \ell( -K ) - \frac{3 \sigma(K) }{4}.
\]
Conversely, if the value of $g_4(K)$ is known then we can use this result to give an upper bound for $\ell(-K)$:
\[
\ell(-K) \le g_4(K) + \frac{3 \sigma(K)}{4}.
\]

The invariant $\theta(K)$ is not only a lower bound for the slice genus $g_4(K)$, but also for the genus of any null-homologous surface bounding $K$ in a negative definite $4$-manifold:

\begin{theorem}\label{thm:hslice}
Let $X$ be a smooth, compact, oriented $4$-manifold with boundary $S^3$ and with negative definite intersection form. Suppose also that $H_1(X ; \mathbb{Z})=0$. Let $K \subset S^3$ be a knot. Then for any, smooth, oriented, properly embedded, null-homologous surface $\Sigma \subset X$ bounding $K$, we have
\[
g(\Sigma) \ge \theta(K),
\]
where $g(\Sigma)$ is the genus of $\Sigma$.
\end{theorem}

For a $4$-manifold $X$ with boundary $S^3$, a knot $K \subset S^3$ is said to be {\em $H$-slice} in $X$ if $K$ bounds a smooth, oriented, properly embedded, null-homologous disc $\Sigma \subset X$. More generally, we define the {\em $H$-slice genus of $K$ in $X$} to be the minimum genus of such a null-homologous surface $\Sigma \subset X$ bounding $K$ and we denote it by $g_H(K,X)$. Thus Theorem \ref{thm:hslice} says that $\theta$ provides a lower bound for the $H$-slice genus of any knot in a negative definite $4$-manifold $X$ with boundary $S^3$ and with $H_1(X ; \mathbb{Z})=0$.

Similar bounds on the $H$-slice genus of $K$ in negative definite manifolds have been obtained by different methods. From \cite{os1}, we have that $g_H(K,X) \ge \tau(K)$, where $\tau$ is the Ozsv\'ath--Szab\'o $\tau$-invariant. In the special case that $X = \#^t \overline{\mathbb{CP}^2} \setminus B^4$, there is a similar inequality $g_H(K,X) \ge s(K)/2$ \cite{mmsw}, where $s$ is the Rasmussen $s$-invariant \cite{ras2}. Yet another bound is given by the Levine--Tristram signature \cite{cn}, which in the case of a negative defininte $4$-manifold with $H_1(X;\mathbb{Z})=0$ gives
\[
g_H(K,X) \ge -\sigma_K(\omega)/2
\]
for all $\omega \in S^1_{!}$. Here $\sigma_K(\omega)$ is the Levine--Tristram signature of $K$ and $S^1_{!}$ is the set of unit complex numbers that are not zeros of any integral coefficient Laurent polynomial $p$ with $p(1)=1$. Note that $S^1_{!}$ contains all roots of unity of prime power order. In particular, since the signature of $K$ is given by $\sigma(K) = \sigma_K(-1)$, we have $g_H(K,X) \ge -\sigma(K)/2$.

We have examples where the genus bound given by $\theta(K)$ is better than that given by $\tau$, $s$ or $\sigma_K(\omega)$. Indeed the knots $K = -9_{42} \# Wh(T_{2,2a+1})$ are an example of this. In fact, since these knots satisfy $\theta(K) = 2 = g_4(K)$, we see that the minimal genus of any null-homologous, properly embedded surface bounding $K$ in any negative definite $4$-manifold with zero integral first homology is exactly $2$.

\subsection{Structure of the paper}

In Section \ref{sec:branch} we collect various results concerning branched covers and spin$^c$-structures that will needed for the main results. In Section \ref{sec:esw} we recall briefly the details of the equivariant Seiberg--Witten--Floer cohomology groups needed for this paper. In Section \ref{sec:conc} we define the concordance invariants $\theta^{(q)}$ and prove various properties of them. In Section \ref{sec:ex} we give some examples of computations of the invariant $\theta$. Finally in Section \ref{sec:definite} we prove the slice genus bounds for surfaces in definite $4$-manifolds bounding a knot $K$.

\subsection{Acknowledgements}

We thank Pedram Hekmati and Hokuto Konno for comments on a draft of this paper.

\section{Branched covers and spin$^c$-structures}\label{sec:branch}

In this section $q$ is allowed to be a prime power. In later sections we will require $q$ to be prime. Let $K \subset S^3$ be a knot in $S^3$. We denote by $Y = \Sigma_q(K)$ the degree $q$ cyclic cover of $S^3$ branched over $K$. Since $q$ is a prime power it follows that $Y$ is a rational homology $3$-sphere \cite[Corollary 3.2]{liv}. Let $\pi : Y \to S^3$ denote the covering map and let $\sigma : Y \to Y$ denote the generator of the natural $\mathbb{Z}_q$-action. From \cite[Corollary 2.2]{jab}, we have that any spin$^c$-structure on $Y \setminus \pi^{-1}(K)$ uniquely extends to $Y$. Now since $H^2( S^3 \setminus K ; \mathbb{Z}) = 0$, there is a unique spin$^c$-structure on $S^3 \setminus K$. The pullback of this spin$^c$-structure under $\pi$ extends uniquely to a spin$^c$-structure on $Y$. Following \cite{jab}, we denote this spin$^c$-structure by $\mathfrak{s}_0 = \mathfrak{s}_0(K , q)$ and we call $\mathfrak{s}_0$ the distinguished spin$^c$-structure on $Y$. We note here that $\mathfrak{s}_0$ is $\sigma$-invariant. To see this first note that the restriction of $\mathfrak{s}_0$ to $Y \setminus \pi^{-1}(K)$ is $\sigma$-invariant because it is a pullback from $S^3 \setminus K$. Then since spin$^c$-structures extend uniquely over $\pi^{-1}(K)$, we get that $\mathfrak{s}_0$ is $\sigma$-invariant.

Now we consider branched covers of $4$-manifolds with boundary. We will be concerned specifically with the case that $X$ is a compact, oriented, smooth $4$-manifold whose boundary is either $S^3$ or two copies of $S^3$. In the latter case, we view one copy of $S^3$ as an ingoing boundary and the other as an outgoing boundary so that $X$ is a cobordism from $S^3$ to $S^3$. Suppose $\Sigma \subset X$ is a smooth, connected, oriented, properly embedded surface which meets the outgoing boundary in a knot $K_1$ and meets the ingoing boundary (if there is one) in a knot $K_0$.

We assume that $H_1(X ; \mathbb{Z}) = 0$. Then $H_2( X , \partial X ; \mathbb{Z}) \cong H^2(X ; \mathbb{Z})$ is torsion-free.

\begin{lemma}\label{lem:meridian}
Suppose that $H_1(X ; \mathbb{Z}) = 0$. If $[\Sigma] \in H_2(X , \partial X ; \mathbb{Z})$ is trivial, then $H_1(X \setminus \Sigma ; \mathbb{Z}) \cong \mathbb{Z}$. If $[\Sigma] \in H_2(X , \partial X ; \mathbb{Z})$ is $k$ times a primitive element, where $k \ge 1$, then $H_2( X \setminus \Sigma ; \mathbb{Z}) \cong \mathbb{Z}/k\mathbb{Z}$. In either case $H_2(X \setminus \Sigma ; \mathbb{Z})$ is generated by a meridian around $\Sigma$.
\end{lemma}
\begin{proof}
From the long exact sequence of the pair $(X , X \setminus \Sigma)$ we obtain an exact sequence
\[
H_2(X ; \mathbb{Z}) \to H_2(X , X \setminus \Sigma ; \mathbb{Z}) \buildrel \partial \over \longrightarrow H_1(X \setminus \Sigma ; \mathbb{Z}) \to H_1(X ; \mathbb{Z}) = 0.
\]
Let $N$ be a tubular neighbourhood of $\Sigma$, which we identify with the open unit disc bundle in the normal bundle of $\Sigma$. By excision and the Thom isomorphism, $H_2(X , X \setminus \Sigma ; \mathbb{Z}) \cong H_2( N , N \setminus \Sigma ; \mathbb{Z}) \cong \mathbb{Z}[ D ]$, where $D$ denotes the unit disc in a fibre of the normal bundle. The boundary of $[D]$ is a meridian, so we see that $H_1(X \setminus \Sigma ; \mathbb{Z})$ is generated by a meridian. It remains to determine the image of the map $H_2(X ; \mathbb{Z}) \to H_2(X , X \setminus \Sigma ; \mathbb{Z})$. Any class $u \in H_2(X ; \mathbb{Z})$ can be represented by a compact oriented embedded surface $S$ in the interior of $X$, meeting $\Sigma$ transversally. The image of $u = [S]$ in $H_2(X , X \setminus \Sigma ; \mathbb{Z})$ under the excision isomorphism is easily seen to be $\langle u , [\Sigma] \rangle [D]$, where $\langle u , [\Sigma] \rangle$ denotes the intersection pairing between $H_2(X ; \mathbb{Z})$ and $H_2(X , \partial X ; \mathbb{Z})$. This is a dual pairing, hence the image of $H_2(X ; \mathbb{Z}) \to H_2(X , X \setminus \Sigma ; \mathbb{Z})$ is exactly $k[D]$ if $[\Sigma]$ is $k$ times a primitive element and is zero is $[\Sigma] = 0$.
\end{proof}

Assume that $H_1(X ; \mathbb{Z}) = 0$ and that $[\Sigma]$ is divisible by $q$. By Lemma \ref{lem:meridian}, it follows that there is a unique homomorphism $H_1(X \setminus \Sigma ; \mathbb{Z}) \to \mathbb{Z}_q$ sending the meridian to $1 \; ({\rm mod} \; q)$. Hence there is a uniquely determined cyclic branched cover $\pi : W \to X$ of degree $q$ and with branching set $\Sigma$. We let $\sigma : W \to W$ denote the generator of the natural $\mathbb{Z}_q$-action. If $\partial X = S^3$ and $\Sigma$ meets $\partial S^3$ in the knot $K_1$, then $\partial W = \Sigma_q(K_1)$ and the $\mathbb{Z}_q$-action on $W$ extends the $\mathbb{Z}_q$-action on $\Sigma_q(K_1)$. If $\partial X = S^3 \cup S^3$ and $\Sigma$ meets the ingoing and outgoing boundaries in $K_0$ and $K_1$, then $W$ is a $\mathbb{Z}_q$-equivariant cobordism from $\Sigma_q(K_0)$ to $\Sigma_q(K_1)$.

Given a knot $K$, let $\sigma_K(\omega)$ denote the Levine--Tristram signature of $K$ and define
\[
\sigma^{(q)}(K) = \sum_{j=1}^{q-1} \sigma_{K}( e^{2\pi i j/q}).
\]

Suppose that $\Sigma \subset D^4$ is a smooth, oriented, properly embedded surface bounding $K$ and let $W \to D^4$ be the $q$-fold cyclic branched cover of $D^4$ branched over $\Sigma$. Then the signature of $W$ equals $\sigma^{(q)}(K)$ \cite{kau}. We note here that when $q$ is odd, $\sigma^{(q)}(K)$ is divisible by $4$ because of the symmetry $\sigma_{K}(\omega) = \sigma_{K}(\omega^{-1})$ of the Levine--Tristram signature.

\begin{lemma}\label{lem:branch}
Let $X$ be a compact, oriented, smooth $4$-manifold with $H_1(X ; \mathbb{Z}) = 0$ and whose boundary is either $S^3$ or two copies of $S^3$. Let $\Sigma \subset X$ be a smooth, connected, oriented, properly embedded surface of genus $g$ which meets the outgoing boundary in a knot $K_1$ and meets the ingoing boundary (if there is one) in a knot $K_0$. Suppose that $[\Sigma] \in H_2(X , \partial X ; \mathbb{Z})$ is a multiple of a prime power $q$. Let $\pi : W \to X$ be the degree $q$ cyclic cover of $X$ branched over $\Sigma$. Then $b_1(W) = 0$. If $\partial X = S^3$, then
\begin{align*}
b_2(W) &=  q b_2(X) + (q-1)(2g), \\
\sigma(W) &= q \sigma(X) - \frac{(q^2-1)}{3q} [\Sigma]^2 + \sigma^{(q)}(K_1), \\
b_+(W) &= q b_+(X) + (q-1)g - \frac{(q^2-1)}{6q}[\Sigma]^2 + \frac{\sigma^{(q)}(K_1)}{2} , \\
\quad b_-(W) &= q b_-(X) + (q-1)g + \frac{(q^2-1)}{6q}[\Sigma]^2 - \frac{\sigma^{(q)}(K_1)}{2}.
\end{align*}
and if $\partial X$ is two copies of $S^3$, then
\begin{align*}
b_2(W) &=  q b_2(X) + (q-1)(2g), \\
\sigma(W) &= q \sigma(X) - \frac{(q^2-1)}{3q} [\Sigma]^2 + \sigma^{(q)}(K_1) - \sigma^{(q)}(K_0), \\
b_+(W) &= q b_+(X) + (q-1)g - \frac{(q^2-1)}{6q}[\Sigma]^2 + \frac{\sigma^{(q)}(K_1)}{2} - \frac{\sigma^{(q)}(K_0)}{2} , \\
\quad b_-(W) &= q b_-(X) + (q-1)g + \frac{(q^2-1)}{6q}[\Sigma]^2 - \frac{\sigma^{(q)}(K_1)}{2} + \frac{\sigma^{(q)}(K_0)}{2}.
\end{align*}

\end{lemma}
\begin{proof}
We give the proofs in the case that $\partial X = S^3$. The case that $\partial X$ has two components is similar.

Choose a surface $\Sigma_1 \subset D^4$ bounding $K_1$. Let $X'$ be the closed $4$-manifold $X' = X \cup_{S^3} \overline{D^4}$, and let $\Sigma' \subset X'$ be the closed embedded surface given by $\Sigma' = \Sigma \cup_{K_1} \overline{\Sigma_1}$. Let $W_1 \to D^4$ be the degree $q$ cyclic covering of $D^4$ branched over $\Sigma_1$. Recall that $\sigma(W_1)$ is given by $\sigma^{(q)}(K_1)$. Let $\pi : W' \to X'$ be the degree $q$ cyclic cover of $X'$ branched over $\Sigma'$. Thus $W' = W \cup_{Y} \overline{W_1}$, where $Y = \Sigma_q(K_1)$. Since $Y$ is a rational homology $3$-sphere, the Mayer--Vietoris sequence implies that $H_1(W ; \mathbb{Q}) \to H_1(W' ; \mathbb{Q})$ is injective. But $b_1(W') = 0$ by \cite{ro}, hence $b_1(W) = 0$ as well.

Next we compute the second Betti number of $W$. Since $\pi : W \to X$ is a degree $q$ covering branched over $\Sigma$, we clearly have
\[
\chi(W) = q \chi(X) - (q-1) \chi(\Sigma).
\]
But since $b_1(X) = b_1(W) = 0$ and $\chi(\Sigma) = 1-2g$, we get
\[
b_2(W) = q b_2(X) + (q-1)(2g).
\]
Now consider the signature of $W$. Novikov additivity gives $\sigma(W') = \sigma(W) - \sigma(W_1) = \sigma(W) - \sigma^{(q)}(K_1)$. On the other hand the $G$-signature theorem for $G = \mathbb{Z}_q$ acting on $W'$ gives
\[
\sigma(X) = \sigma(X') = \frac{\sigma(W')}{q} + \frac{q^2-1}{3q} [\widetilde{\Sigma'}]^2
\]
where $\widetilde{\Sigma'}$ is the preimage of $\Sigma'$ under $\pi : W' \to X'$. The degree of the normal bundle of $\Sigma'$ in $X'$ is $q$ times the degree of the normal bundle of $\widetilde{\Sigma'}$ in $W'$, so $[\widetilde{\Sigma'}]^2 = [\Sigma']^2/q = [\Sigma]^2/q$ and so
\[
\sigma(W') = q \sigma(X) - \frac{(q^2-1)}{3q} [\Sigma]^2.
\]
Thus
\[
\sigma(W) = q \sigma(X) - \frac{(q^2-1)}{3q} [\Sigma]^2 + \sigma^{(q)}(K_1).
\]
Using $b_{\pm}(W) = (b_2(W) \pm \sigma(W))/2$, we get
\[
b_+(W) = q b_+(X) + (q-1)g - \frac{(q^2-1)}{6q}[\Sigma]^2 + \frac{\sigma^{(q)}(K_1)}{2}
\]
and
\[
b_-(W) = q b_-(X) + (q-1)g + \frac{(q^2-1)}{6q}[\Sigma]^2 - \frac{\sigma^{(q)}(K_1)}{2}.
\]

\end{proof}

Next we turn our attention to spin$^c$-structures on $W$. Let $X$ and $\pi : W \to X$ be as in Lemma \ref{lem:branch}.

\begin{proposition}\label{prop:spinc}
For any spin$^c$-structure $\mathfrak{s}_X$ on $X$, there exists a $\sigma$-invariant spin$^c$-structure $\mathfrak{s}$ on $W$ such that the restriction of $\mathfrak{s}$ to $W \setminus \pi^{-1}(\Sigma)$ is the pullback of $\mathfrak{s}_X$ restricted to $X \setminus \Sigma$. Moreover, the restriction of $\mathfrak{s}$ to any boundary component $\Sigma_q(K_i)$ of $W$ coincides with the distinguished spin$^c$-structure $\mathfrak{s}_0$.
\end{proposition}
\begin{proof}
We will construct $\mathfrak{s}$ as an equivariant spin$^c$-structure, by which we mean that we will construct the corresponding spinor bundles and equip them with a $\mathbb{Z}_q$-action lifting $\sigma$. Let $\widetilde{\Sigma} = \pi^{-1}( \Sigma)$ be the fixed point set of $\sigma$. Let $S_X^{\pm} \to X$ denote the spinor bundles on $X$ corresponding to the spin$^c$-structure $\mathfrak{s}_X$. Let $S_0^{\pm}$ denote the restriction of $S_X^{\pm}$ to $X \setminus \Sigma$ and let $\widetilde{S_0}^{\pm}$ denote the pullbacks of $S_0^{\pm}$ to $W \setminus \widetilde{\Sigma}$ under $\pi$. Since $\pi : W \setminus \widetilde{\Sigma} \to X \setminus \Sigma$ is an unbranched cover, $\widetilde{S_0}^{\pm}$ define a spin$^c$-structure $\mathfrak{s}'$ on $W \setminus \widetilde{\Sigma}$. Furthermore, since $\widetilde{S_0}^{\pm}$ are obtained by pullback, they carry a natural action of $\mathbb{Z}_q$ which makes $\mathfrak{s}'$ into an equivariant spin$^c$-structure.

We will show that $\mathfrak{s}'$ extends equivariantly over $\widetilde{\Sigma}$. Choose a $\sigma$-invariant tubular neighbourhood $\widetilde{N}$ of $\widetilde{\Sigma}$, which we may identify with the normal bundle of $\widetilde{\Sigma}$. Let $N = \pi(\widetilde{N})$ be the image of $\widetilde{N}$ under $\pi$. Then $N$ is a tubular neighbourhood of $\Sigma$ which can be identified with the normal bundle of $\Sigma$. Then we have an unbranched covering $\widetilde{N} \setminus \widetilde{\Sigma} \to N \setminus \Sigma$.

Choose a $\sigma$-invariant metric $g$ on $W$. Then we obtain a $\sigma$-equivariant isomorphism $T(\widetilde{N}) \cong T\widetilde{\Sigma} \oplus \widetilde{N}$. This is a direct sum of $\sigma$-equivariant complex line bundles, so has a $\sigma$-equivariant complex structure. Associated to this complex structure is $\sigma$-equivariant spin$^c$-structure which we denote by $\mathfrak{s}_{\widetilde{N}}$. If we can show that the restriction of $\mathfrak{s}_{\widetilde{N}}$ to $\widetilde{N} \setminus \widetilde{\Sigma}$ is equivariantly isomorphic to the restriction of $\mathfrak{s}'$ to $\widetilde{N} \setminus \widetilde{\Sigma}$, then it follows that $\mathfrak{s}'$ extends equivariantly over $\widetilde{\Sigma}$, since we can glue $\mathfrak{s}'$ and $\mathfrak{s}_{\widetilde{N}}$ together over $\widetilde{N} \setminus \widetilde{\Sigma}$ using this isomorphism.

Now since $\sigma$ acts freely on $\widetilde{N} \setminus \widetilde{\Sigma}$ with quotient space $N \setminus \Sigma$, we have that any equivariant spin$^c$-structure on $\widetilde{N} \setminus \widetilde{\Sigma}$ descends to a spin$^c$-structure on $N \setminus \Sigma$. Moreover any two equivariant spin$^c$-structures on $\widetilde{N} \setminus \widetilde{\Sigma}$ are equivariantly isomorphic if and only if the corresponding descended spin$^c$-structures on $N \setminus \Sigma$ are isomorphic.

Consider the restriction of $\mathfrak{s}'$ to $\widetilde{N} \setminus \widetilde{\Sigma}$. This is the pullback of $\mathfrak{s}_X |_{N \setminus \Sigma}$, hence it descends to $\mathfrak{s}_X |_{N \setminus \Sigma}$. Note that this spin$^c$-structure extends to $N$ as $\mathfrak{s}_X |_N$. Next consider the restriction of $\mathfrak{s}_{\widetilde{N}}$ to $\widetilde{N} \setminus \widetilde{\Sigma}$. Since $N = \widetilde{N}/\sigma$, it is easily seen that $\mathfrak{s}_{\widetilde{N}}$ descends to the spin$^c$-structure on $N \setminus \Sigma$ associated to the complex structure on $T(N \setminus \Sigma) \cong T\Sigma \oplus N$. Since this isomorphism extends over $\Sigma$ to an isomorphism $T(N) \cong T\Sigma \oplus N$, it follows that the spin$^c$-structure obtained by descending $\mathfrak{s}_{\widetilde{N}}$ to $N \setminus \Sigma$ extends to a spin$^c$-structure on $N$.

Lastly, since $N$ deformation retracts to $\Sigma$ and $H^2(\Sigma ; \mathbb{Z}) = 0$ (because $\Sigma$ has non-empty boundary), it follows that there is a unique spin$^c$-structure on $N$. Hence the spin$^c$-structures on $N \setminus \Sigma$ obtained by descending $\mathfrak{s}'$ and $\mathfrak{s}_{\widetilde{N}}$ to $N \setminus \Sigma$ are isomorphic, since they both extend to the unique spin$^c$-structure on $N$. This proves the claim that $\mathfrak{s}'$ extends equivariantly to a spin$^c$-structure on $W$, which we denote by $\mathfrak{s}$. Since $\mathfrak{s}$ is an equivariant spin$^c$-structure, its isomorphism class is $\sigma$-invariant.

It remains to show that the restriction of $\mathfrak{s}$ to any boundary component $\Sigma_q(K_i)$ of $W$ coincides with the distinguished spin$^c$-structure $\mathfrak{s}_0$. From the construction of $\mathfrak{s}$, it follows that the restriction of $\mathfrak{s}$ to $\Sigma_q(K_i) \setminus \pi^{-1}(K_i)$ is the pullback of the unique spin$^c$-structure on $S^3 \setminus K_i$. But this property uniquely characterises the distinguished spin$^c$-structure $\mathfrak{s}_0$, so the restriction of $\mathfrak{s}$ to $\Sigma_q(K_i)$ is the distinguished spin$^c$-structure.
\end{proof}

\begin{lemma}\label{lem:cohom}
The pullback $\pi^* : H^2(X ; \mathbb{Z}) \to H^2(W ; \mathbb{Z})$ is injective and the image is precisely the $\sigma$-invariant elements of $H^2(W ; \mathbb{Z})$.
\end{lemma}
\begin{proof}
For $a \in H^2(X ; \mathbb{Z})$, $\pi_*(\pi^* a) = qa$. Injectivity of $\pi^*$ follows, since $H^2(X ; \mathbb{Z})$ is torsion free.

To prove that the image of $\pi^*$ is precisely the $\sigma$-invariant elements of $H^2(W ; \mathbb{Z})$ we will compute the degree $2$ equivariant cohomology of $W$ in two ways. First using the Leray--Serre spectral sequence for the Borel fibration. This spectral sequence has $E_2^{p,q} = H^p( BG ; H^q(W ; \mathbb{Z}))$ where $G = \mathbb{Z}_q$ and abuts to the equivariant cohomology $H^*_G(W ; \mathbb{Z})$. Since $b_1(W) = 0$, we have $H^1(W ; \mathbb{Z}) = 0$ and so the only non-zero terms in $E_2^{p,q}$ with $p+q=2$ are $E_2^{0,2} = H^2(W ; \mathbb{Z})^G$ and $E_2^{2,0} = H^2(BG ; \mathbb{Z}) = H^2_G(pt ; \mathbb{Z}) \cong \mathbb{Z}_q$. Furthermore $E_2^{0,1} = E_2^{2,1} = E_2^{3,0} = 0$, so there are no non-zero differentials out of $E_r^{0,2}$ or in to $E_r^{2,0}$ for any $r$. So we get an exact sequence 
\[
0 \to H^2_G( pt ; \mathbb{Z} ) \to H^2_G(W ; \mathbb{Z}) \to H^2(W ; \mathbb{Z})^G \to 0.
\]
Furthermore, the action of $G$ on $W$ has fixed points. Taking a fixed point $w \in W$ gives a splitting
\[
H^2_G(W ; \mathbb{Z}) \cong H^2(W ; \mathbb{Z})^G \oplus H^2_G( w ; \mathbb{Z} )
\]
and it follows that we have an isomorphism
\[
H^2_G(W , w ; \mathbb{Z}) \cong H^2(W ; \mathbb{Z})^G.
\]
Next, we compute the degree $2$ equivariant cohomology using the Mayer--Vietoris sequence. Let $G$ act trivially on $X$ so that $\pi : W \to X$ may be regarded as an equivariant map. Consider the open cover of $X$ by $X \setminus \Sigma$, $N$, where $N$ is a tubular neighbourhood of $\Sigma$. Let $\widetilde{\Sigma} = \pi^{-1}(\Sigma)$ and $\widetilde{N} = \pi^{-1}(N)$. Then $W \setminus \widetilde{\Sigma}$, $\widetilde{N}$ is an open $G$-invariant cover of $W$. Since the cover of $W$ is obtained by pulling back the cover of $X$ we get a commutative diagram between Mayer--Vietors sequences (all cohomology groups having coefficients in $\mathbb{Z}$):
\[
\xymatrix@C-=0.4cm{
H^1_G(\widetilde{\Sigma}) \! \oplus \! H^1_G(W \! \setminus \! \widetilde{\Sigma}) \ar[r]^-{\tilde{i}} \! & \!  H^1_G(\widetilde{N} \!  \setminus \! \widetilde{\Sigma})  \ar[r] \! & \! H^2_G(W) \ar[r] \! & \! H^2_G(\widetilde{\Sigma}) \! \oplus \! H^2_G(W \! \setminus \! \widetilde{\Sigma}) \ar[r]^-{\tilde{j}} \! & \! H^2_G(\widetilde{N} \! \setminus \! \widetilde{\Sigma})  \\
H^1_G(\Sigma) \! \oplus \! H^1_G(X \! \setminus \! \Sigma) \ar[r]^-{i} \ar[u]_-{\pi^*} \! & \!  H^1_G(N \!  \setminus \! \Sigma) \ar[r] \ar[u]_-{\pi^*} \! & \! H^2_G(X) \ar[r] \ar[u]_-{\pi^*} \! & \! H^2_G(\Sigma) \! \oplus \! H^2_G(X \! \setminus \! \Sigma) \ar[r]^-{j} \ar[u]_-{\pi^*} \! & \! H^2_G(N \! \setminus \! \Sigma) \ar[u]_-{\pi^*}
}
\]
We claim that $\pi^*$ induces isomorphisms $coker( \tilde{i} ) \cong coker(i)$ and $ker(\tilde{j}) \cong ker(j)$. These claims imply that $\pi^* : H^2_G(X) \to H^2_G(W)$ is an isomorphism. Consider first $i$ and $\tilde{i}$. We have that $\pi : \widetilde{\Sigma} \to \Sigma$ is a homeomorphism. Also $G$ acts freely on $\widetilde{W} \setminus \widetilde{\Sigma}$ and $\widetilde{N} \setminus \widetilde{\Sigma}$ with quotient spaces $X \setminus \Sigma$ and $N \setminus \Sigma$. Furthermore $H^1_G(pt ; \mathbb{Z}) = 0$ so it follows that the first two vertical maps in the above diagram are isomorphisms. This shows that $\pi^* : coker(i) \to coker(\tilde{i})$ is an isomorphism. Similarly the last two vertical maps fit into a commutative diagram with exact columns:
\[
\xymatrix@R-=0.6cm{
0 & 0 \\
H^2_G(\widetilde{\Sigma}) \oplus H^2_G(W \setminus \widetilde{\Sigma}) \ar[u] \ar[r]^-{\tilde{j}} & H^2_G(\widetilde{N} \setminus \widetilde{\Sigma}) \ar[u] \\
H^2_G(\Sigma) \oplus H^2_G(X \setminus \Sigma) \ar[u]_-{\pi^*} \ar[r]^-{j} & H^2_G(N \setminus \Sigma) \ar[u]_-{\pi^*} \\
H^2_G(pt) \ar[u]_-{(0 , \iota )} \ar[r]^-{\cong} & H^2_G(pt) \ar[u]_-{\iota} \\
0 \ar[u] & 0 \ar[u]
}
\]
where $\iota : H^2_G(pt) \to H^2_G(X \setminus \Sigma)$ and $\iota : H^2_G(pt) \to H^2(N \setminus \Sigma)$ are the maps induced by $X \setminus \Sigma \to pt$ and $N \setminus \Sigma \to pt$. It follows that $\pi^* : ker(j) \to ker(\tilde{j})$ is an isomorphism.

We have proven that $\pi^* : H^2_G(X ; \mathbb{Z}) \to H^2_G(W ; \mathbb{Z})$ is an isomorphism. Let $w \in W$ be a fixed point and $x = \pi(w)$. Then we get an isomorphism $\pi^* : H^2( X , x ; \mathbb{Z}) \to H^2(W , w ; \mathbb{Z})$. But $H^2_G(X,x ; \mathbb{Z}) \cong H^2(X ; \mathbb{Z})$ via the forgetful map from equivariant to non-equivariant cohomology and similarly $H^2_G(W , w ; \mathbb{Z}) \cong H^2( W )^G$ via the forgetful map (corresponding to the $p=0$ column of the spectral sequence $E_r^{p,q}$). Hence $\pi^* : H^2(X ; \mathbb{Z}) \to H^2(W ; \mathbb{Z})^G$ is an isomorphism.
\end{proof}

Since each boundary component of $X$ is a $3$-sphere, we have $H^2(X ; \mathbb{Z})$ $\cong$ \linebreak $H^2(X , \partial X ; \mathbb{Z})$ and that the intersection form on $H^2(X , \partial X ; \mathbb{Z})$ is unimodular. Furthermore $H_1(X ; \mathbb{Z}) = 0$ implies that $H^2(X ; \mathbb{Z})$ and $H^2(X , \partial X ; \mathbb{Z})$ are torsion-free. So spin$^c$-structures on $X$ are in bijection with characteristics $c \in H^2(X ; \mathbb{Z})$, that is, elements such that $\langle c , x \rangle = \langle x , x \rangle \; ({\rm mod} \; 2)$ for every $x \in H^2(X , \partial X ; \mathbb{Z})$. Now we can give an improved version of Proposition \ref{prop:spinc}.

\begin{proposition}\label{prop:spinc2}
If $q$ is odd, then for any characteristic $c \in H^2(X ; \mathbb{Z})$, there is a $\sigma$-invariant spin$^c$-structure $\mathfrak{s}$ on $W$ such that $c_1(\mathfrak{s}) = \pi^*(c)$ in $H^2(W ; \mathbb{Q})$ and such that $\mathfrak{s}$ restricts to the distinguished spin$^c$-structure on each component of $\partial W$.

If $q=2$, then for any characteristic $c \in H^2(X ; \mathbb{Z})$, there is a $\sigma$-invariant spin$^c$-structure $\mathfrak{s}$ on $W$ such that $c_1(\mathfrak{s}) = \pi^*(c + [\Sigma]/2)$ in $H^2(W ; \mathbb{Q})$ and such that $\mathfrak{s}$ restricts to the distinguished spin$^c$-structure on each component of $\partial W$.
\end{proposition}
\begin{proof}
Let $q$ be odd. Let $\mathfrak{s}_X$ be any spin$^c$-structure on $X$. Then from Proposition \ref{prop:spinc}, there exists a $\sigma$-invariant spin$^c$-structure $\mathfrak{s}$ on $W$ which restricts to $\mathfrak{s}_0$ on each component of $\partial W$. Then $c_1(\mathfrak{s})$ is $\sigma$-invariant, so by Lemma \ref{lem:cohom} we have $c_1(\mathfrak{s}) = \pi^*(c)$ for some $c \in H^2(X ; \mathbb{Z})$. We claim that $c$ is a characteristic for $H^2(X ; \mathbb{Z})$. To see this, let $a \in H^2(X ; \mathbb{Z})$. Then using the fact that $c_1(\mathfrak{s})$ is a characteristic, we get
\begin{align*}
q \langle c , a \rangle &= \langle \pi^*(c) , \pi^*(a) \rangle \\
&= \langle c_1(\mathfrak{s}) , \pi^*(a) \rangle \\
&= \langle \pi^*(a) , \pi^*(a) \rangle \; ({\rm mod} \; 2) \\
&= q \langle a,a \rangle \; ({\rm mod} \; 2).
\end{align*}
Since $q$ is odd, we deduce that $\langle c , a \rangle = \langle a,a \rangle \; ({\rm mod} \; 2)$, so $c$ is a characteristic. Now given any other characteristic $d \in H^2(X ; \mathbb{Z})$, we have that $d = c + 2u$ for some $u \in H^2(X ; \mathbb{Z})$. Let $L \to X$ be a line bundle with $c_1(L) = u$. Then $\mathfrak{s}_L = L \otimes \mathfrak{s}$ is a $\sigma$-invariant spin$^c$-structure on $W$ with $c_1(\mathfrak{s}_L) = c_1(\mathfrak{s}) + 2 \pi^*( c_1(L) ) = \pi^*( c + 2u) = \pi^*(d)$. Also, $L$ restricted to the boundary of $X$ is trivial (since $H^2(S^3 ; \mathbb{Z}) = 0$), so $\mathfrak{s}_L$ is isomorphic to $\mathfrak{s}$ on the boundary of $W$ and hence $\mathfrak{s}_L$ restricts to the distinguished spin$^c$-structure on each component of $\partial W$.

Now consider the case $q=2$. As in the odd case, let $\mathfrak{s}_X$ be any spin$^c$-structure on $X$. Then from Proposition \ref{prop:spinc}, there exists a $\sigma$-invariant spin$^c$-structure $\mathfrak{s}$ on $W$ which restricts to $\mathfrak{s}_0$ on each component of $\partial W$. Then $c_1(\mathfrak{s})$ is $\sigma$-invariant, so by Lemma \ref{lem:cohom} we have that $c_1(\mathfrak{s}) = \pi^*(y)$ for some $y \in H^2(X ; \mathbb{Z})$. Set $c = y - [\Sigma]/2$ so that $c_1(\mathfrak{s}) = \pi^*(c + [\Sigma]/2 )$. 

We claim that $c$ is a characteristic for $H^2(X ; \mathbb{Z})$. From \cite[Lemma 3.4]{nag} we have $w_2(W) = \pi^*( w_2(X) ) + [\widetilde{\Sigma}] = \pi^*( w_2(X) + [\Sigma]/2 )$. Let $c' \in H^2( X ; \mathbb{Z})$ be a characteristic. Then $c' = w_2(X) \; ({\rm mod} \; 2)$ and hence $\pi^*( c' + [\Sigma]/2) = w_2(W) \; ({\rm mod} \; 2)$. So $\pi^*(c' + [\Sigma]/2 )$ is a characteristic for $H^2(W ; \mathbb{Z})$. But $\pi^*(c + [\Sigma]/2)$ is also a characteristic, so it follows that $\pi^*(c) = \pi^*(c') \; ({\rm mod} \; 2)$. So $\pi^*(c) = \pi^*(c') + 2w$ for some $w \in H^2(W ; \mathbb{Z})$. Applying $\pi_*$ we get $2c = 2c' + 2\pi_*(w)$. But $H^2(X ; \mathbb{Z})$ is torsion free, so $c = c' + \pi_*(w)$. Applying $\pi^*$ and using $\pi^*(\pi_*(w)) = w + \sigma^*(w)$, we see that $\pi^*(c) = \pi^*(c') + w + \sigma^*(w)$. Comparing this with $\pi^*(c) = \pi^*(c') + 2w$, we see that $\sigma^*(w) = w$. So by Lemma \ref{lem:cohom}, we have that $w = \pi^*(v)$ for some $v \in H^2(X ; \mathbb{Z}$. Hence $\pi^*(c) = \pi^*(c' + 2v)$. By Lemma \ref{lem:cohom}, $\pi^*$ is injective and hence $c = c' + 2v$. This proves the claim that $c$ is a characteristic for $H^2(X ; \mathbb{Z})$, because $c'$ is a characteristic for $H^2(X ; \mathbb{Z})$.

The rest of the argument works the same as the odd case. Given any other characteristic $d \in H^2(X ; \mathbb{Z})$, we have that $d = c + 2u$ for some $u \in H^2(X ; \mathbb{Z})$. Let $L \to X$ be a line bundle with $c_1(L) = u$. Then $\mathfrak{s}_L = L \otimes \mathfrak{s}$ is a $\sigma$-invariant spin$^c$-structure on $W$ with $c_1(\mathfrak{s}_L) = c_1(\mathfrak{s}) + 2 \pi^*( c_1(L) ) = \pi^*( c + 2u + [\Sigma]/2) = \pi^*(d + [\Sigma]/2)$. Also, $L$ restricted to the boundary of $X$ is trivial, so $\mathfrak{s}_L$ restricts to the distinguished spin$^c$-structure on each component of $\partial W$.
\end{proof}

\section{Equivariant Seiberg--Witten--Floer cohomology and equivariant $d$-invariants}\label{sec:esw}

Let $q$ be a prime and let $G$ be the cyclic group $G  = \mathbb{Z}_q$. Let $\sigma$ denote a generator of $G$. Let $Y$ be a rational homology $3$-sphere and $\mathfrak{s}$ a spin$^c$-structure on $Y$. Suppose that $\mathbb{Z}_q$ acts on $Y$ by orientation preserving diffeomorphisms and that $\mathfrak{s}$ is invariant under this action. In \cite{bh}, the author and Hekmati constructed the equivariant Seiberg--Witten--Floer cohomology groups $HSW^*_G(Y , \mathfrak{s})$ (the construction works more generally for any finite group action, but we will only need the case that $G$ is cyclic of prime order). Up to a degree shift, $HSW^*_G(Y , \mathfrak{s})$ is equal to the $S^1 \times G$-equivariant cohomology of the Conley index $I(Y , \mathfrak{s})$ of a finite dimensional approximation of the Chern--Simons--Dirac gradient flow. We take cohomology with respect to the coefficient group $\mathbb{F} = \mathbb{Z}_q$. Thus $HSW^*_G(Y , \mathfrak{s})$ is a module over the ring $R = H^*_{S^1 \times G}( pt ; \mathbb{F})$. If $q=2$, then $R \cong \mathbb{F}[U,Q]$, where $deg(U) = 2$, $deg(Q) = 1$. If $q$ is odd, then $R \cong \mathbb{F}[U,R,S]/(R^2)$, where $deg(U) = 2$, $deg(R)=1$, $deg(S)=2$.

The localisation theorem in equivariant cohomology implies that the localisation $U^{-1} HSW^*_G(Y , \mathfrak{s})$ is a free $U^{-1}R$-module of rank $1$. Thus we have an isomorphism of the form
\[
\iota : U^{-1}HSW^*_G(Y , \mathfrak{s}) \to U^{-1}R \tau
\]
for some element $\tau$. We then define a sequence of invariants, which are to be thought of as an equivariant analogue of the Ozsv\'ath--Szab\'o $d$-invariant $d(Y , \mathfrak{s})$. These are defined as follows. If $q=2$, we set $d_{G , Q^j}(Y , \mathfrak{s}) = i-j$, where $i$ is the least degree for which there exists an element $x \in HSW^i_G(Y , \mathfrak{s})$ and a $k \in \mathbb{Z}$ such that
\[
\iota x = Q^j U^k \tau \; ({\rm mod} \; Q^{j+1})
\]
(cf \cite[\textsection 3.6]{bh}). If $q$ is odd, we set $d_{G , S^j}(Y , \mathfrak{s}) = i - 2j$, where $i$ is the least degree for which there exists an element $x \in HSW^i_G(Y , \mathfrak{s})$ and a $k \in \mathbb{Z}$ such that
\[
\iota x = S^j U^k \tau \; ({\rm mod} \; S^{j+1})
\]
For convenience, we also define equivariant $\delta$-invariants which are related to the $d$-invariants by
\[
\delta_{G , Q^j}(Y , \mathfrak{s}) = \frac{1}{2} d_{G , Q^j}(Y , \mathfrak{s})
\]
if $q=2$ and 
\[
\delta_{G , S^j}(Y , \mathfrak{s}) = \frac{1}{2} d_{G , S^j}(Y , \mathfrak{s})
\]
if $q$ is odd.

Most of the properties of the $d$-invariants that we need are proven in \cite{bh}, however we also need a result concerning equivariant connected sums.

Let $Y_1,Y_2$ be rational homology $3$-spheres. Suppose that $G = \mathbb{Z}_q$ acts smoothly and orientation preservingly on $Y_1,Y_2$ and that the action has non-empty fixed point sets. Since $G$ acts orientation preservingly, the fixed points sets $F_1,F_2$ must be $1$-dimensional. For $i=1,2$, let $N_i$ denote the normal bundle of $F_i$ in $Y_i$. Let $y_1 \in Y_1$, $y_2 \in Y_2$ be fixed points. Assume that the action of $G$ on the normal spaces $(N_1)_{y_1}$, $(N_2)_{y_2}$ are isomorphic as representations of $G$ by an orientation preserving isomorphism $\varphi$. Then a neighbourhood of $y_i$ in $Y_i$ takes the form $(-1,1) \times (N_i)_{y_i}$ where $G$ acts trivially on the first factor. Then $(t,x) \mapsto (-t,\varphi(x))$ defines a $G$-equivariant orientation reversing diffeomorphism from $(-1,1) \times (N_1)_{y_1}$ to $(-1,1) \times (N_2)_{y_2}$. Thus by removing $G$-invariant neighbourhoods of $y_1,y_2$ and identifying their boundaries via this map, we can form the $G$-equivariant connected sum $Y = Y_1 \# Y_2$.

Let $\mathfrak{s}_1, \mathfrak{s}_2$ be $G$-invariant spin$^c$-structures on $Y_1,Y_2$ and set $\mathfrak{s} = \mathfrak{s}_1 \# \mathfrak{s}_2$.
\begin{proposition}\label{prop:dsum}
For all $i,j \ge 0$, we have
\[
d_{G , Q^{i+j}}(Y , \mathfrak{s}) \le d_{G , Q^i}(Y_1 , \mathfrak{s}_1) + d_{G , Q^j}(Y_2 , \mathfrak{s}_2)
\]
if $q=2$, and
\[
d_{G , S^{i+j}}(Y , \mathfrak{s}) \le d_{G , S^i}(Y_1 , \mathfrak{s}_1) + d_{G , S^j}(Y_2 , \mathfrak{s}_2).
\]
if $q$ is odd.
\end{proposition}
\begin{proof}
We give the proof in the $q=2$ case. The case where $q$ is odd is similar. In the $q=2$ case we have $R \cong \mathbb{F}[U,Q]$. For an equivariant connected sum, one has that the Conley indices are related by $I(Y,\mathfrak{s}) \cong I(Y_1 , \mathfrak{s}_1) \wedge I(Y_2 , \mathfrak{s}_2)$ and this defines a product map on cohomology groups 
\begin{equation}\label{equ:product}
\mu : HSW^i_{G}(Y_1 , \mathfrak{s}_1) \otimes_{\mathbb{F}[U,Q]} HSW^j_{G}(Y_2 , \mathfrak{s}_2) \to HSW^{i+j}_{\mathbb{Z}_2}(Y , \mathfrak{s}).
\end{equation}
Recall that we have localisation isomorphisms
\[
\iota_1 : U^{-1}HSW^*_G(Y_1 , \mathfrak{s}_1) \to \mathbb{F}[U,U^{-1} , Q]\tau_1
\]
and
\[
\iota_2 : U^{-1}HSW^*_G(Y_2 , \mathfrak{s}_2) \to \mathbb{F}[U,U^{-1} , Q]\tau_2
\]
for some $\tau_1,\tau_2$. One can show that $I(Y , \mathfrak{s})^{S^1} \cong I(Y_1 , \mathfrak{s}_1)^{S^1} \wedge I(Y_2 , \mathfrak{s}_2)^{S^1}$ and from this it follows that the localisation of the product map (\ref{equ:product}) is an isomorphism. So letting $\iota : U^{-1}HSW^*_G(Y , \mathfrak{s}) \to \mathbb{F}[U,U^{-1} , Q]\tau$ denote the localisation map corresponding to $(Y , \mathfrak{s})$, it follows that we must have
\[
\iota( \mu( \iota_1^{-1}(\tau_1) \otimes \iota_2^{-1}(\tau_2) ) ) = c U^k \tau \; ({\rm mod} \; Q )
\]
for some $c \in \mathbb{F} \setminus \{0\}$ and some $k \in \mathbb{Z}$.

Now set $a_1 = d_{G , Q^i}(Y_1 , \mathfrak{s}_1) + i$, $a_2 = d_{G , Q^j}(Y_2 , \mathfrak{s}_2) + j$. Then by definition of the $d$-invariants there exists $x_1 \in HSW^{a_1}_G(Y_1 , \mathfrak{s}_1)$ such that $\iota_1 x_1 = Q^i  U^{k_1} \tau_1 \; ({\rm mod} \; Q^{i+1} )$ and $\iota_2 x_2 = Q^j U^{k_2} \tau_2 \; ({\rm mod} \; Q^{j+1})$ for some $k_1,k_2$. It follows that
\[
\iota( \mu( c^{-1}x_1 , x_2 ) ) = Q^{i+j} U^{k_1+k_2+k} \tau \; ({\rm mod} \; Q^{i+j+1} ).
\]
But $\mu( c^{-1}x_1 , x_2 ) \in HSW^{a_1+a_2}_G(Y , \mathfrak{s})$, so from the definition of $d_{G , Q^{i+j}}(Y , \mathfrak{s})$, we get
\[
d_{G , Q^{i+j}}(Y , \mathfrak{s}) \le a_1 + a_2 - i - j = d_{G , Q^i}(Y_1 , \mathfrak{s}_1) + d_{G , Q^j}(Y_2 , \mathfrak{s}_2).
\]
\end{proof}

\section{Concordance invariants from equivariant Seberg--Witten--Floer cohomology}\label{sec:conc}

Let $K \subset S^3$ be a knot. Let $\sigma(K)$ and $g_4(K)$ denote the signature and smooth $4$-genus of $K$. In \cite{bh}, the author and Hekmati constructed a sequence of integer-valued knot concordance invariants $\{ \delta_j(K) \}_{j \ge 0}$ with the following properties (see \cite[\textsection 6]{bh}):

\begin{itemize}
\item[(1)]{$\delta_0(K) \ge \delta(K)$, where $\delta(K)$ is the Manolescu--Owens invariant \cite{mo}.}
\item[(2)]{$\delta_{j+1}(K) \le \delta_j(K)$ for all $j \ge 0$.}
\item[(3)]{$\delta_j(K) \ge -\sigma(K)/2$ for all $j \ge 0$ and $\delta_j(K) = -\sigma(K)/2$ for $j \ge g_4(K)-\sigma(K)/2$.}
\item[(4)]{$\delta_j(K) + \delta_j(-K) \ge 0$ for all $j \ge 0$.}
\item[(5)]{$\delta_j(K) = -\sigma(K)/2 \; ({\rm mod} \; 4)$ for all $j \ge 0$.}
\item[(6)]{If $K$ is quasi-alternating, then $\delta_j(K) = -\sigma(K)/2$ for all $j \ge 0$.}
\end{itemize}

The invariant $\delta_j(K)$ is defined by
\[
\delta_j(K) = 2 d_{ \mathbb{Z}_2 , Q^j }( \Sigma_2(K) , \mathfrak{s}_0) = 4 \delta_{\mathbb{Z}_2 , Q^j}( \Sigma_2(K) , \mathfrak{s}_0).
\]
For convenience we define a new set of concordance invariants $\xi_j(K)$ by setting
\[
\xi_j(K) = \frac{1}{4} \delta_j(K) + \frac{1}{8}\sigma(K).
\]
Then the above properties of $\delta_j(K)$ imply that $\xi_j(K)$ is integer-valued, decreasing and is zero for $j \ge g_4(K) -\sigma(K)/2$.

\begin{proposition}\label{prop:xi}
Let $K_+,K_-$ be knots where $K_-$ is obtained from $K_+$ by changing a positive crossing into a negative crossing. Then $\sigma(K_-) - \sigma(K_+) = 0$ or $2$, and:
\begin{itemize}
\item[(1)]{If $\sigma(K_-) = \sigma(K_+)$, then $\xi_{j+1}(K_-) \le \xi_j(K_+) \le \xi_j(K_-)$ for all $j \ge 0$.}
\item[(2)]{If $\sigma(K_-) = \sigma(K_+)+2$, then $\xi_{j+1}(K_+) \le \xi_j(K_-) \le \xi_j(K_+)$ for all $j \ge 0$.}
\end{itemize}
\end{proposition}
\begin{proof}
Resolving the singularity of a regular homotopy associated to the crossing change gives a genus $1$ surface $\Sigma$ properly embedded in $X = [0,1] \times S^3$ which meets the ingoing boundary of $X$ in $K_-$ and the outgoing boundary in $K_+$. Let $W \to X$ be the double cover of $X$ branched over $\Sigma$. This gives a $\mathbb{Z}_2$-equivariant cobordism from $\Sigma_2(K_-)$ to $\Sigma_2(K_+)$ where $\mathbb{Z}_2$ acts by covering transformation. By Proposition \ref{prop:spinc2}, we can choose a $\mathbb{Z}_2$-invariant spin$^c$-structure $\mathfrak{s}$ on $W$ which restricts to the distinguished spin$^c$-structure on either boundary and has $c_1(\mathfrak{s}) = 0$. By Lemma \ref{lem:branch}, we find $\sigma(W) = \sigma(K_+) - \sigma(K_-)$ and $b_2(W)=2$. Thus
\[
b_+(W) = 1 + \frac{ \sigma(K_+) - \sigma(K_-) }{2}.
\]
Recall that for a single crossing change one has $\sigma(K_+) - \sigma(K_-) \in \{ 0 ,-2\}$. If $\sigma(K_+) = \sigma(K_-)$, then $b_+(W) = b_-(W) = 1$. In this case the equivariant Froyshov inequality \cite[Theorem 5.3]{bh} applied to $(W , \mathfrak{s})$ gives
\[
-\frac{\sigma(W)}{2} + \delta_{j+1}(K_-) \le \delta_j(K_+).
\]
Since $\sigma(W) = \sigma(K_+)  - \sigma(K_-)$, this can be rewritten as
\[
\xi_{j+1}(K_-) \le \xi_j(K_+).
\]
Next, consider the case $\sigma(K_+) = \sigma(K_-) - 2$. Then $b_+(W) = 0$ and $b_-(W) = 2$. In this case the equivariant Froyshov inequalities for $W$ and $\overline{W}$ give
\[
\xi_j(K_-) \le \xi_j(K_+).
\]

Consider again a regular homotopy associated to the crossing change. This gives an immersed surface in $[0,1] \times S^3$ which meets the ingoing boundary in $K_-$, the outgoing boundary in $K_+$ and has a single transverse point of self-intersection with self-intersection $-1$. Instead of resolving the crossing, we take a blow-up. This gives a genus zero surface $\Sigma$ properly embedded in $X$, where $X$ is $\mathbb{CP}^2$ with two balls removed and $\Sigma$ meets the ingoing boundary of $X$ in $K_-$ and the outgoing boundary in $K_+$. Let $\pi : W \to X$ be the double cover of $X$ branched over $\Sigma$. We again obtain a $\mathbb{Z}_2$-equivariant cobordism from $\Sigma_2(K_-)$ to $\Sigma_2(K_+)$. By Proposition \ref{prop:spinc2}, we can choose a $\mathbb{Z}_2$-invariant spin$^c$-structure $\mathfrak{s}$ on $W$ such that $\mathfrak{s}$ restricts to the distinguished spin$^c$-structure on each boundary and $c_1(\mathfrak{s})^2 = 2$ (take $c$ in Proposition \ref{prop:spinc2} to be a characteristic satisfying $c^2 = 1$). Lemma \ref{lem:branch} gives $\sigma(W) = \sigma(K_+) - \sigma(K_-) + 2$ and $b_2(W) = 2$.

Note that $H^+( \overline{W})^{\mathbb{Z}_2} \cong H^-(W)^{\mathbb{Z}_2} \cong H^-(X) = 0$. Hence reversing orientation on $W$ and swapping ingoing and outgoing boundaries, we may apply the Froyshov inequality. If $\sigma(K_+) = \sigma(K_-)$, then $b_+(W) = 2$, $b_-(W) = 0$ and the equivariant Froyshov inequality applied to $\overline{W}$ gives:
\[
\frac{\sigma(W) - c_1(\mathfrak{s})^2}{2} + \delta_{j}(K_+) \le \delta_j(K_-).
\]
Since $\sigma(W) = \sigma(K_+) - \sigma(K_-) - 2$ and $c_1(\mathfrak{s})^2 = -2$, this simplifies to
\[
\frac{\sigma(K_+)}{2} - \frac{\sigma(K_-)}{2} + \delta_{j}(K_+) \le \delta_j(K_-)
\]
which can be re-written as
\[
\xi_{j}(K_+) \le \xi_j(K_-).
\]
Similarly if $\sigma(K_+) = \sigma(K_-) - 2$, then applying the equivariant Froyshov inequality to $\overline{W}$ gives:
\[
\xi_{j+1}(K_+) \le \xi_j(K_-).
\]
Putting together all these inequalities, we see that if $\sigma(K_-) = \sigma(K_+)$, then $\xi_{j+1}(K_-) \le \xi_j(K_+) \le \xi_j(K_-)$ for all $j \ge 0$ and if $\sigma(K_-) = \sigma(K_+)+2$, then $\xi_{j+1}(K_+) \le \xi_j(K_-) \le \xi_j(K_+)$ for all $j \ge 0$.
\end{proof}

Extend the definition of $\xi_j(K)$ to negative $j$ by setting $\xi_j(K) = \xi_0(K)$ for all $j<0$. Consider the shifted invariants $\rho_j(K) = \xi_{j - \sigma(K)/2}(K)$. Then $\{ \rho_j(K) \}$ is integer-valued, decreasing and $\rho_j(K) = 0$ for $j \ge g_4(K)$ by property (3) of the $\delta_j$-invariants. We define $\theta(K)$ to be the smallest $j\ge 0$ for which $\rho_j( -K)=0$. Equivalently, $\theta(K) = \max\{ 0 , j(-K)-\sigma(K)/2\}$ where $j(-K)$ is the smallest $j$ such that $\xi_j(-K)=0$. Since $\rho_j$ are knot concordance invariants, it follows that $\theta$ is also a concordance invariant.

\begin{theorem}\label{thm:theta}
The invariant $\theta(K)$ satisfies the following properties:
\begin{itemize}
\item[(1)]{$-\sigma(K)/2 \le \theta(K) \le g_4(K)$.}
\item[(2)]{$\theta(K_1 + K_2) \le \theta(K_1) + \theta(K_2)$.}
\item[(3)]{Let $K_+,K_-$ be knots where $K_-$ is obtained from $K_+$ by changing a positive crossing into a negative crossing. Then
\[
0 \le \theta(K_+) - \theta(K_-) \le 1.
\]
}
\item[(4)]{If $K$ is quasi-alternating, then 
\[
\theta(K) = \begin{cases} -\sigma(K)/2, & \sigma(K) \le 0, \\ 0, & \sigma(K) > 0. \end{cases}
\]
}
\item[(5)]{If $\delta(K) < -\sigma(K)/2$ and $\sigma(K) \le 0$, then $\theta(K) \ge 1 -\sigma(K)/2$.
}
\end{itemize}
\end{theorem}
\begin{proof}
The inequality $\theta(K) \ge -\sigma(K)/2$ follows from the definition $\theta(K) = \max\{ 0 , j(-K)-\sigma(K)/2\}$ and the fact that $j(-K) \ge 0$.

The inequality $\theta(K) \le g_4(K)$ follows from the definition of $\theta(K)$ and the fact that $j(-K) \le g_4(K)+\sigma(K)/2$ (which in turn follows from the fact that $\xi_j(-K) = 0$ for $j \ge g_4(K) + \sigma(K)/2$).

For (2), we note that $Y = \Sigma_2(K_1 + K_2)$ is (equivariantly) diffeomorphic to the connected sum of $Y_1 = \Sigma_2(K_1)$ and $Y_2 = \Sigma_2(K_2)$. Furthermore $\mathfrak{s}_0(K_1 \# K_2) = \mathfrak{s}_0(K_1) \# \mathfrak{s}_0(K_2)$ \cite[Theorem 2.4]{jab}. Then from Proposition \ref{prop:dsum}, we have
\[
d_{\mathbb{Z}_2 , Q^{i+j} }( \Sigma_2(K_1 + K_2) , \mathfrak{s}_0 ) \le d_{\mathbb{Z}_2 , Q^i}( \Sigma_2(K_1) , \mathfrak{s}_0) + d_{\mathbb{Z}_2 , Q^j}( \Sigma_2(K_2) , \mathfrak{s}_0)
\]
for all $i,j \ge 0$. It follows that $\delta_{i+j}(K_1 + K_2) \le \delta_i(K_1) + \delta_j(K_2)$ for all $i,j \ge 0$. Hence $j(K_1 + K_2) \le j(K_1) + j(K_2)$. Similarly $j(-K_1 - K_2 ) \le j(-K_1) + j(-K_2)$. From this and additivity of the signature, we get $\theta(K_1 + K_2 ) \le \theta(K_1) + \theta(K_2)$.

For (3), we consider separately the cases $\sigma(K_+) = \sigma(K_-)$ and $\sigma(K_+) = \sigma(K_-) - 2$. If $\sigma(K_+) = \sigma(K_-)$, then by Proposition \ref{prop:xi} (1), we have
\[
\xi_{j+1}(K_-) \le \xi_j(K_+) \le \xi_j(K_-).
\]
Setting $j = j(K_-)$, we have $\xi_j(K_-) = \xi_{j+1}(K_-) = 0$, hence $\xi_j(K_+) = 0$. So $j(K_+) \le j(K_-)$. If $j(K_-) > 1$, then $\xi_{j-1}(K_-) > 0$ and hence $\xi_{j-2}(K_+) \ge \xi_{j-1}(K_-) > 0$, so $j(K_+) \ge j(K_-) - 1$. Thus we have $j(K_-)-1 \le j(K_+) \le j(K_-)$. Consider the mirrors $-K_+$ and $-K_-$. Then $-K_+$ is obtained from $-K_-$ by changing a positive crossing to a negative crossing and $\sigma(-K_-) = \sigma(-K_+)$, so we get $j(-K_+)-1 \le j(-K_-) \le j(-K_+)$. Hence
\begin{equation}\label{equ:jbar}
j(-K_+) - \frac{\sigma(K_+)}{2} - 1 \le j(-K_-) - \frac{\sigma(K_-)}{2} \le j(-K_+) - \frac{\sigma(K_+)}{2}
\end{equation}
from which is follows easily that $0 \le \theta(K_+) - \theta(K_-) \le 1$. In the case $\sigma(K_+) - \sigma(K_-) = -2$, Proposition \ref{prop:xi} (2) gives
\[
\xi_{j+1}(K_+) \le \xi_j(K_-) \le \xi_j(K_+).
\]
Arguing as above this gives, $j(K_+)-1 \le j(K_-) \le j(K_+)$. A similar argument applied to the mirrors of $K_+$ and $K_-$ gives $j(-K_-)-1 \le j(-K_+) \le j(-K_-)$. Using this and $\sigma(K_-) = \sigma(K_+) + 2$, we again find that the inequality (\ref{equ:jbar}) holds and hence $0 \le \theta(K_+) - \theta(K_-) \le 1$.

For (4), if $K$ is quasi-alternating then $\delta_j(-K) = -\sigma(-K)/2$ for all $j \ge 0$. Hence $j(-K) = 0$ and $\theta(K) = \max \{ -\sigma(K)/2 , 0 \}$.

For (5), if $\delta(K) < -\sigma(K)/2$, then $\delta_0(-K) \ge \delta(-K) > -\sigma(-K)/2$. Hence $\xi_0(-K) > 0$ and $j(-K) \ge 1$. Hence $\theta(K)  \ge j(-K) - \sigma(K)/2 \ge 1 - \sigma(K)/2$.
\end{proof}

\begin{corollary}\label{cor:unk}
For any knot $K$ we have $\theta(K) + \theta(-K) \le u(K)$ where $u(K)$ is the unknotting number of $K$.
\end{corollary}
\begin{proof}
By Theorem \ref{thm:theta} (3), each crossing change can increase $\theta(K) + \theta(-K)$ by at most one. Hence by considering a sequence of $u(K)$ crossing changes starting at the unknot and finishing with $K$, we get $\theta(K) + \theta(-K) \le u(K)$.
\end{proof}

\begin{proposition}
For prime knots with $9$ or fewer crossings we have that $\theta(K) = \nu^+(K)$ and $\theta(K) = \nu^+(-K)$, except for $K=9_{42}$. For $K = 9_{42}$ we have $\theta(K) = \nu^+(K) = \nu^+(-K) = 0$ and $\theta(-K) = 1$. 
\end{proposition}
\begin{proof}
All prime knots with $9$ or fewer crossings are quasi-alternating except for $8_{19}, 9_{42}$ and $9_{46}$, so we only need to check these three cases. The knot $K=8_{19}$ has $-\sigma(K)/2 = \tau(K) = g_4(K) = u(K) = 3$. Hence $\theta(K) = \nu^+(K) = 3$. Then since $\theta(K) + \theta(-K) = 3 + \theta(-K) \le u(K)=3$, we must have $\theta(-K) = 0$. Similarly, $\nu^+(K) + \nu^+(-K) \le 3$, so $\nu^+(-K) = 0$.

The knot $K = 9_{42}$ has $\sigma(K)/2 = g_4(K) = u(K) = 1$. Hence $1 = \sigma(K)/2 \le \omega(-K) \le g_4(K) = 1$, giving $\theta(-K) = 1$. Then since $\theta(K) + \theta(-K) \le u(K) = 1$, we have $\theta(K) = 0$. On the other hand, $\tau(K) = 0$. Then from \cite[Theorem 1.1]{sasa}, it follows that $\nu^+(K) = \nu^+(-K) = 0$.

Lastly, the knot $K=9_{46}$ is slice, so $\theta(K) = \theta(-K) = \nu(K) = \nu^+(-K) = 0$.
\end{proof}

So far we have obtained concordance invariants related to the branched double cover $\Sigma_2(K)$. We now generalise the above results by considering analogous concordance invariants associated to cyclic branched covers of any prime order.

Given a knot $K$ in $S^3$ and a prime number $q$, we let $Y = \Sigma_q(K)$ denote the degree $q$ cyclic cover of $S^3$ branched over $K$. As discussed in Section \ref{sec:branch}, there is a distinguished spin$^c$-structure $\mathfrak{s}_0$ on $Y$. In \cite{jab}, Jabuka defined a series of knot concordance invariants by definining
\[
\delta^{(q)}(K) = 4 \delta( \Sigma_q(K) , \mathfrak{s}_0 ).
\]
Jabuka's invariants are defined more generally for prime powers but we restrict here to the prime case because this is necessary in order to apply the equivariant Froyshov inequality \cite[Theorem 5.3]{bh}. Note also that Jabuka used the notation $\delta_q(K)$ for what we have denoted $\delta^{(q)}(K)$. 

For any odd prime $q$ and any $j \ge 0$, we define a knot concordance invariant $\delta_j^{(q)}(K)$ by setting
\[
\delta_j^{(q)}(K) = 4 \delta_{\mathbb{Z}_q , S^j}( \Sigma_q(K) , \mathfrak{s}_0),
\]
where $\delta_{\mathbb{Z}_q , S^j}$ are the equivariant $\delta$-invariants, as defined in \cite{bh}. We also write $\delta_j^{(2)}(K) = \delta_j(K)$ for the case $q=2$.

\begin{proposition}
Let $q$ be an odd prime. The concordance invariants $\delta_j^{(q)}(K)$ satisfy the following properties:
\begin{itemize}
\item[(1)]{$\delta^{(q)}_0(K) \ge \delta^{(q)}(K)$, where $\delta^{(q)}(K)$ is the Jabuka invariant \cite{jab}.}
\item[(2)]{$\delta^{(q)}_{j+1}(K) \le \delta^{(q)}_j(K)$ for all $j \ge 0$.}
\item[(3)]{$\delta^{(q)}_j(K) \ge -\sigma^{(q)}(K)/2$ for all $j \ge 0$ and $\delta^{(q)}_j(K) = -\sigma^{(q)}(K)/2$ for $2j \ge (q-1)g_4(K)-\sigma^{(q)}(K)/2$.}
\item[(4)]{$\delta^{(q)}_j(K) + \delta^{(q)}_j(-K) \ge 0$ for all $j \ge 0$.}
\item[(5)]{$\delta^{(q)}_j(K) = -\sigma^{(q)}(K)/2 \; ({\rm mod} \; 4)$ for all $j \ge 0$.}
\end{itemize}
\end{proposition}
\begin{proof}
The proof is a straightforward extension of the $q=2$ case \cite[Proposition 6.2]{bh} to an arbitrary prime, so we omit the details.
\end{proof}

For convenience we set
\[
\xi_j^{(q)}(K) = \frac{1}{4} \delta^{(q)}_j(K) + \frac{1}{8}\sigma^{(q)}(K).
\]
Then $\xi^{(q)}_j(K)$ is integer-valued, decreasing as a function of $j$ and is zero for $2j \ge (q-1)g_4(K) -\sigma^{(q)}(K)/2$.

As in the $q=2$ case, we have the following properties:

\begin{proposition}\label{prop:xi2}
Let $K_+,K_-$ be knots where $K_-$ is obtained from $K_+$ by changing a positive crossing into a negative crossing. Then $-(q-1) \le \sigma^{(q)}(K_+) - \sigma^{(q)}(K_-) \le 0$ or $2$ and for all $j \ge 0$, we have
\[
\xi^{(q)}_{j + \alpha}(K_-) \le \xi_j^{(q)}(K_+) \text{ and } \xi^{(q)}_{j+ \beta } \le \xi^{(q)}_j(K_-),
\]
where $\alpha = (q-1)/2 + \sigma^{(q)}(K_+)/4 - \sigma^{(q)}(K_-)/4$ and $\beta = \sigma^{(q)}(K_-)/4 - \sigma^{(q)}(K_+)/4$.
\end{proposition}
\begin{proof}
The proof is a straightforward extension of the proof of Proposition \ref{prop:xi}, so we give only a sketch. Resolving the singularity of a regular homotopy associated to the crossing change gives a genus $1$ surface $\Sigma$ properly embedded in $X = [0,1] \times S^3 \times$ which meets the ingoing boundary of $X$ in $K_-$ and the outgoing boundary in $K_+$. Let $W \to X$ be the $q$-fold cover of $X$ branched over $\Sigma$. This gives a $\mathbb{Z}_q$-equivariant cobordism from $\Sigma_q(K_-)$ to $\Sigma_q(K_+)$. One finds that $b_2(W) = 2(q-1)$ and that $\sigma(W) = \sigma^{(q)}(K_+) - \sigma^{(q)}(K_-)$, hence $b_+(W) = (q-1) + \sigma^{(q)}(K_+)/2 - \sigma^{(q)}(K_-)/2 = 2\alpha$. Note that $2\alpha = b_+(W) \ge 0$, hence $\sigma^{(q)}(K_+) - \sigma^{(q)}(K_-) \ge -(q-1)$. Applying the equivariant Froyshov inequality to $W$ gives $\xi^{(q)}_{j+\alpha}(K_-) \le \xi^{(q)}_j(K_+)$.

Consider again a regular homotopy associated to the crossing change. Instead of resolving the crossing, we take a blow-up. This gives a genus zero surface $\Sigma$ properly embedded in $X$, where $X$ is $\mathbb{CP}^2$ with two balls removed and $\Sigma$ meets the ingoing boundary of $X$ in $K_-$ and the outgoing boundary in $K_+$. Let $\pi : W \to X$ be the $q$-fold cover of $X$ branched over $\Sigma$. We find $b_2(W) = q$ and $\sigma(W) = q + \sigma^{(q)}(K_+) - \sigma^{(q)}(K_-)$, hence $b_-(W) = \sigma^{(q)}(K_-)/2 - \sigma^{(q)}(K_+)/2 = 2\beta$. Since $2\beta = b_-(W) \ge 0$, we have $\sigma^{(q)}(K_+) - \sigma^{(q)}(K_-) \le 0$. Applying the Froyshov inequality to $\overline{W}$ (with $\Sigma_q(K_+)$ regarded as an ingoing boundary and $\Sigma_q(K_-)$ an outgoing boundary gives $\xi^{(q)}_{j+\beta}(K_+) \le \xi^{(q)}_j(K_-)$.
\end{proof}

For an odd prime $q$, we define
\[
\theta^{(q)}(K) = \max\left\{ 0 , \dfrac{  2 j^{(q)}(-K)}{q-1}-\dfrac{\sigma^{(q)}(K)}{2(q-1)}\right\},
\]
where $j^{(q)}(-K)$ is the smallest $j$ such that $\xi^{(q)}_j(-K)=0$. Note that by this definition $\theta^{(q)}(K)$ is valued in $\frac{1}{(q-1)}\mathbb{Z}$. 

\begin{theorem}\label{thm:thetaq}
The invariant $\theta^{(q)}(K)$ satisfies the following properties:
\begin{itemize}
\item[(1)]{$-\dfrac{\sigma^{(q)}(K)}{2(q-1)} \le \theta^{(q)}(K) \le g_4(K)$.}
\item[(2)]{$\theta^{(q)}(K_1 + K_2) \le \theta^{(q)}(K_1) + \theta^{(q)}(K_2)$.}
\item[(3)]{Let $K_+,K_-$ be knots where $K_-$ is obtained from $K_+$ by changing a positive crossing into a negative crossing. Then
\[
0 \le \theta^{(q)}(K_+) - \theta^{(q)}(K_-) \le 1.
\]
}
\item[(4)]{If $\Sigma_q(K)$ is an $L$-space, then 
\[
\theta^{(q)}(K) = \begin{cases} -\dfrac{\sigma^{(q)}(K)}{2(q-1)}, & \sigma^{(q)}(K) \le 0, \\ 0, & \sigma^{(q)}(K) > 0. \end{cases}
\]
}
\item[(5)]{If $\delta^{(q)}(K) < -(q-1)\sigma^{(q)}(K)/(2(q-1))$ and $\sigma^{(q)}(K) \le 0$, then $\theta^{(q)}(K) \ge 1/(q-1) -\sigma^{(q)}(K)/(2(q-1))$.
}
\end{itemize}
\end{theorem}
\begin{proof}
We prove only part (3), since the other parts are similar to the $q=2$ case. First note that for any knot $K$, we have
\begin{equation}\label{equ:thetabar}
(q-1)\theta^{(q)}(-K) = \max\{ 0 , 2j^{(q)}(K) + \sigma^{(q)}(K)/2 \}.
\end{equation}
Now let $K_-$ be obtained from $K_+$ by changing a positive crossing into a negative crossing. Then from Proposition \ref{prop:xi2} we have
\[
\xi^{(q)}_{j + \alpha}(K_-) \le \xi_j^{(q)}(K_+) \text{ and } \xi^{(q)}_{j+ \beta } \le \xi^{(q)}_j(K_-),
\]
where $2\alpha = (q-1) + \sigma^{(q)}(K_+)/2 - \sigma^{(q)}(K_-)/2$ and $2\beta = \sigma^{(q)}(K_-)/2 - \sigma^{(q)}(K_+)/2$. Now if $\xi^{(q)}_j(K_+) = 0$, then $\xi^{(q)}_{j+\alpha}(K_-) \le 0$, hence $\xi^{(q)}_{j+\alpha}(K_-) = 0$. It follows that $j^{(q)}(K_-) \le j^{(q)}(K_+) + \alpha$. Similarly, if $\xi^{(q)}_j(K_-) = 0$, then $\xi^{(q)}_{j+\beta}(K_+) \le 0$, so we deduce that $j^{(q)}(K_+) \le j^{(q)}(K_-) + \beta$. Putting these together, we have
\[
j^{(q)}(K_+) - \beta \le j^{(q)}(K_-) \le j^{(q)}(K_+) + \alpha.
\]
From the definitions of $\alpha$ and $\beta$, this is equivalent to
\begin{equation}\label{equ:jq}
2 j^{(q)}(K_+) +\frac{\sigma^{(q)}(K_+)}{2} \le 2 j^{(q)}(K_-) + \frac{\sigma^{(q)}(K_-)}{2} \le 2 j^{(q)}(K_+) + \frac{\sigma^{(q)}(K_+)}{2} + (q-1).
\end{equation}
If $2j^{(q)}(K_+) +\sigma^{(q)}(K_+)/2 \ge 0$, then we also have $2j^{(q)}(K_-) +\sigma^{(q)}(K_-)/2 \ge 0$. So from (\ref{equ:thetabar}), we see that $(q-1)\theta^{(q)}(-K_+) = 2 j^{(q)}(K_+) +\sigma^{(q)}(K_+)/2$ and $\theta^{(q)}(-K_-) = 2j^{(q)}(K_-) +\sigma^{(q)}(K_-)/2$. Then from (\ref{equ:jq}), we get
\[
(q-1)\theta^{(q)}(-K_+) \le (q-1) \theta^{(q)}(-K_-) \le (q-1)\theta^{(q)}(-K_+) + (q-1).
\]

If $j^{(q)}(K_+) +\sigma^{(q)}(K_+)/2 \le 0$, then $\theta^{(q)}(-K_+) = 0$ and from (\ref{equ:jq}) we get
\[
(q-1)\theta^{(q)}(-K_-) \le 2 j^{(q)}(K_-) +\frac{\sigma^{(q)}(K_-)}{2} \le 2j^{(q)}(K_+) +\frac{\sigma^{(q)}(K_+)}{2} + (q-1) \le (q-1),
\]
hence 
\[
0 \le (q-1) \theta^{(q)}(-K_-) \le (q-1).
\]
So in either case, we get that
\begin{equation}\label{equ:thetabar2}
\theta^{(q)}(-K_+) \le \theta^{(q)}(-K_-) \le \theta^{(q)}(-K_+) + 1.
\end{equation}
Lastly, we note that if $K_-$ is obtained from $K_+$ by changing a positive crossing into a negative crossing, then $-K_+$ is obtained from $-K_-$ by changing a positive crossing to a negative crossing. Replacing $K_+$ by $-K_-$ and $K_-$ by $-K_+$ in (\ref{equ:thetabar2}), we get
\[
\theta^{(q)}(K_-) \le \theta^{(q)}(K_+) \le \theta^{(q)}(K_-) + 1
\]
or equivalently $0 \le \theta^{(q)}(K_+) - \theta^{(q)}(K_-) \le 1$.
\end{proof}

Using the same argument as the $q=2$ case, we have:

\begin{corollary}
For any knot $K$ we have $\theta^{(q)}(K) + \theta^{(q)}(-K) \le u(K)$.
\end{corollary}

We consider a more general concordance invariant $\theta^{(q)}(K,m)$ depending also on a non-negative integer $m$.

\begin{definition}
Let $m$ be a non-negative integer and $q$ a prime. For any knot $K$, define $j^{(q)}(K,m)$ to be the least value of $j$ such that $\delta_j^{(q)}(K) \le m -\sigma^{(q)}(K)/2$. Further, define $\omega^{(q)}(K,m) \in \frac{1}{(q-1)}\mathbb{Z}$ by
\[
\theta^{(2)}(K,m) = \max\left\{ 0 ,  j^{(2)}(-K,m) - \frac{\sigma(K)}{2} \right\}
\]
for $q=2$, and
\[
\theta^{(q)}(K,m) = \max\left\{ 0 , \frac{ 2 j^{(q)}(-K,m)}{q-1} - \frac{\sigma^{(q)}(K)}{2(q-1)} \right\}
\]
for odd $q$.
\end{definition}

\begin{proposition}\label{prop:est}
Let $\ell^{(q)}(K)$ denote the lowest $i$ for which $HF^+_i( \Sigma_q(K) , \mathfrak{s}_0 )$ is non-zero (where $HF^+$ is taken with coefficients in $\mathbb{F} = \mathbb{Z}/q\mathbb{Z}$). Then
\[
j^{(2)}(K,m) \ge \ell^{(2)}(K) - \frac{m}{2} + \frac{\sigma(K)}{4}
\]
and
\[
2 j^{(q)}(K,m) \ge \ell^{(q)}(K) - \frac{m}{2} + \frac{\sigma^{(q)}(K)}{4}
\]
for odd $q$. Hence for all $q$ we have
\[
\theta^{(q)}(K,m) \ge \frac{\ell^{(q)}(-K)}{(q-1)} - \frac{m}{2(q-1)} - \frac{3 \sigma^{(q)}(K)}{4(q-1)}.
\]
In particular, setting $m=0$, we have
\[
\theta^{(q)}(K) \ge \frac{\ell^{(q)}(-K)}{(q-1)} - \frac{3 \sigma^{(q)}(K)}{4(q-1)}.
\]
\end{proposition}
\begin{proof}
We consider the case where $q$ is odd. The case $q=2$ is similar. Let $Y = \Sigma_q(K)$ and let $\mathfrak{s}_0$ be the distinguished spin$^c$-structure. Then $\delta^{(q)}_j(K)/2 = d_{\mathbb{Z}_q , S^j}(Y , \mathfrak{s}_0)$ is defined to be $k-2j$, where $k$ is the minimum degree of an element of the equivariant Seiberg--Witten--Floer cohomology $HSW^*_{\mathbb{Z}_q}(Y , \mathfrak{s}_0)$ satisfying the condition given in Section \ref{sec:esw}. From this definition it follows that $\delta^{(q)}_j(K) \ge 2(k_0 - 2j)$ where $k_0$ is the lowest degree in which $HSW^*_{\mathbb{Z}_q}(Y , \mathfrak{s}_0)$ is non-zero. From \cite[Theorem 3.2]{bh}, there is a spectral sequence $E_r^{p,q}$ abutting to $HSW^*_{\mathbb{Z}_q}(Y , \mathfrak{s}_0)$ which has $E_2^{p,q} = H^p( B\mathbb{Z}_q ; HSW^q(Y , \mathfrak{s}_0 ) )$. Hence $k_0$ is at least the lowest degree in which $HSW^q(Y , \mathfrak{s}_0)$ is non-zero. Now recall that the Seiberg--Witten--Floer cohomology $HSW^*(Y , \mathfrak{s}_0)$ is isomorphic to the Heegaard--Floer cohomology $HF_+^*(Y , \mathfrak{s}_0)$. Further, since we are working over a field this is also isomorphic to the Heegaard--Floer homology $HF^+_*(Y , \mathfrak{s}_0)$. So the lowest degree in which $HSW^*(Y , \mathfrak{s}_0)$ is non-zero is precisely $\ell^{(q)}(K)$. So we have
\[
\delta^{(q)}_j(K) \ge 2(k_0 - 2j ) \ge 2 \ell^{(q)}(K) - 4j.
\]
Now let $j = j^{(q)}(K,m)$, so $\delta^{(q)}_j(K) \le m -\sigma^{(q)}(K)/2$, giving 
\[
2j^{(q)}(K,m) \ge \ell^{(q)}(K) -\frac{m}{2} + \frac{\sigma^{(q)}(K)}{4}
\]
and thus
\[
2\frac{j^{(q)}(K,m)}{q-1} + \frac{\sigma^{(q)}(K)}{2} \ge  \frac{\ell^{(q)}(K)}{q-1} -\frac{m}{2(q-1)} + \frac{3 \sigma^{(q)}(K)}{4(q-1)}.
\]
Replacing $K$ by $-K$, we thus get
\[
\theta^{(q)}(K,m) \ge \frac{\ell^{(q)}(-K)}{(q-1)} - \frac{m}{2(q-1)} - \frac{3 \sigma^{(q)}(K)}{4(q-1)}.
\]
\end{proof}

\section{Examples}\label{sec:ex}

\subsection{Torus knots}

Let $a,b$ be positive coprime integers and let $T_{a,b}$ denote the $(a,b)$ torus knot.

\begin{proposition}
Let $q$ be coprime to $a$ and $b$. Then $\theta^{(q)}( -T_{a,b} ) = 0$. 
\end{proposition}
\begin{proof}
We have that $\sigma^{(q)}(T_{a,b}) \le 0$ as $T_{a,b}$ is braid positive. From \cite[Proposition 7.2]{bh}, we have $\delta_j^{(q)}( T_{a,b} ) = -\sigma^{(q)}( T_{a,b})/2$ for all $j \ge 0$. So $j^{(q)}( T_{a,b} ) = 0$ and $\theta^{(q)}( -T_{a,b} ) = \max\{ 0 , \sigma^{(q)}(T_{a,b})/(2(q-1)) \} = 0$.
\end{proof}

\begin{proposition}
For each $n \ge 1$, we have $\theta( T_{3,6n - 1} )  = g_4( T_{3,6n - 1}) = 6n-2$ and $\theta(T_{3,6n+1}) = g_4(T_{3,6n+1}) = 6n$.
\end{proposition}
\begin{proof}
From \cite[Proposition 7.3]{bh}, we have that $j(T_{3,6n-1}) = 2n-2$. Also $\sigma(T_{3,6n-1}) = -8n$, hence $\theta(T_{3,6n-1}) = 2n-2 + 4n = 6n-2 = g_4(T_{3,6n-1})$. Similarly from \cite[Proposition 7.4]{bh}, we have that $j(T_{3,6n+1}) = 2n$ and we also have $\sigma(T_{3,6n+1}) = -8n$. Hence $\theta(T_{3,6n+1}) = 2n+4n = 6n= g_4(T_{3,6n+1})$.
\end{proof}

\begin{proposition}
For each $n \ge 1$, we have $\theta^{(3)}(T_{2,6n-1}) = g_4(T_{2,6n-1}) = 3n-1$ and $\theta^{(3)}(T_{2,6n+1}) = g_4(T_{2,6n+1}) = 3n$.
\end{proposition}
\begin{proof}
From \cite[Proposition 7.2]{bh}, we have that $\sigma^{(3)}(T_{2,6n\pm 1}) = \sigma^{(2)}(T_{3,6n \pm 1}) = \sigma(T_{3,6n \pm 1})$, which equals $-8n$. Also from \cite[Proposition 7.3, 7.4]{bh}, we find that $\ell^{(3)}(-T_{2,6n-1}) = -2$ and $\ell^{(3)}(-T_{2,6n+1}) = 0$. Putting this into Proposition \ref{prop:est} gives $\theta^{(3)}(T_{2,6n-1}) \ge 3n-1$ and $\theta^{(3)}(T_{2,6n+1}) \ge 3n$. But we also find that $g_4(T_{2,6n-1}) = 3n-1$, $g_4(T_{2,6n+1}) = 3n$, so we must have equalities $\theta^{(3)}(T_{2,6n-1}) = 3n-1$, $\theta^{(3)}(T_{2,6n+1}) = 3n$.
\end{proof}

In light of the above propositions we suspect that $\theta^{(q)}( T_{a,b} ) = g_4( T_{a,b})$ for all primes $q$ and all positive coprime integers $a,b$.

\begin{proposition}
For each $n\ge 1$ and $m \ge 0$, we have $\theta( T_{3,6n-1} , m ) = \max \{ 4n , 6n-2 - 2\lfloor \frac{m}{4} \rfloor \}$ and $\theta( T_{3,6n+1} , m ) = \max \{ 4n , 6n - 2\lfloor \frac{m}{4} \rfloor \}$.
\end{proposition}
\begin{proof}
Recall that $\sigma( T_{3,6n\pm 1} )  = -8n$. Then $\theta( T_{3,6n \pm 1} , m ) = j( -T_{3,6n\pm 1} , m) + 4n$. From \cite[Proposition 7.3]{bh}, we have that 
\[
\delta_j( -T_{3,6n-1}  ) = \begin{cases} -4\left( \left\lfloor \frac{j}{2} \right\rfloor + 1\right) & 0 \le j \le 2n-3, \\ -4n & j \ge 2n-2. \end{cases}
\]
It follows that $j(-T_{3,6n-1} , m) = \max\{ 0 , 2n-2 - 2 \lfloor \frac{m}{4} \rfloor \}$ and thus $\theta( T_{3,6n-1} , m ) = \max \{ 4n , 6n-2 - 2\lfloor \frac{m}{4} \rfloor \}$. The case of $T_{3,6n+1}$ is similar. From \cite[Proposition 7.4]{bh}, we have that
\[
\delta_j( -T_{3,6n+1}  ) = \begin{cases} -4 \left\lfloor \frac{j}{2} \right\rfloor & 0 \le j \le 2n-1, \\ -4n & j \ge 2n. \end{cases}
\]
Hence $j(-T_{3,6n+1} , m) = \max\{ 0 , 2n - 2\lfloor \frac{m}{4} \rfloor \}$ and thus $\theta( T_{3,6n+1} , m ) = \max \{ 4n , 6n - 2\lfloor \frac{m}{4} \rfloor \}$.
\end{proof}

\subsection{Examples involving Whitehead doubles}

For any knot $K$, let $Wh(K)$ denote the untwisted Whitehead double of $K$ with a positive clasp. Whitehead doubles are known to have trivial Alexander polynomial and are therefore topologically slice \cite{fq}. This also implies that the Levine--Tristram signature is identically zero. A Whitehead double can be unknotted by a single crossing change, so they have $g_4(Wh(K)) \le 1$ and $u(Wh(K)) \le 1$.

\begin{proposition}
If $\tau(K) > 0$, then $\theta( Wh(K) ) = 1$, $\theta(-Wh(K)) = 0$.
\end{proposition}
\begin{proof}
By \cite[Theorem 1.5]{mo}, we have that $\delta( Wh(K) ) < 0$ whenever $\tau(K) > 0$. Then since $\sigma(Wh(K)) = 0$, Theorem \ref{thm:theta} (5) gives $\theta( Wh(K) ) \ge 1$. On the other hand, Corollary \ref{cor:unk} gives $\theta(Wh(K)) + \theta( -Wh(K)) \le u( Wh(K) )  \le 1$. Hence $\theta(Wh(K)) = 1$ and $ \theta(-Wh(K)) = 0$.
\end{proof}

Note that from \cite[Theorem 1.4]{hed}, we have that $\tau( Wh(K) ) = 1$ whenever $\tau(K) > 0$. Hence for such knots we also have $\nu^+(Wh(K)) = 1$, $\nu^+( -Wh(K)) = 0$.

\begin{example}
The Whitehead double $Wh(T_{2,2a+1})$ of the torus knot $T_{2,2a+1}$ has $\tau(Wh(T_{2,2a+1}) ) = 1$, $s( Wh(T_{2,2a+1})) = 2$ and $\delta( Wh(T_{2,2a+1}) ) = -4a$ \cite{mo}. Consider the knot $K = Wh(T_{2,2a+1}) \# (-Wh(T_{2,2b+1}))$. Then $\tau(K) = s(K) = 0$ and the Levine--Tristram signature of $K$ is identically zero. However $\delta(K) = 4(b-a)$. So if $a>b$ then $\theta(K) \ge 1$. On the other hand $\theta(K) \le \theta( Wh(T_{2,2a+1})) + \theta( -Wh(T_{2,2b+1})) = 1$, so we have $\theta(K) = 1$.
\end{example}

\begin{example}
Let $K = -9_{42} \# Wh(T_{2,2a+1})$ where $a$ is any positive integer. Then $\tau(K) = 1$, $s(K) = 2$, $\sigma(K) = -2$ and $\delta(K) = 1-4a < -\sigma(K)/2$. Therefore $\theta(K) \ge 2$. But $9_{42}$ and $Wh(T_{2,2a+1})$ both have slice genus $1$ and unknotting number $1$, hence $g_4(K), u(K) \le 2$. It follows that $\theta(K) = g_4(K) = u(K) = 2$. On the other hand $\tau$, $s$ and $\sigma$ (or more generally the Levine--Tristram signature) of $K$ only give the bounds $g_4(K) \ge 1$ and $u(K) \ge 1$. We also have $\nu^+(K) \le \nu^+(-9_{42}) + \nu^+( Wh(T_{2,2a+1}) ) = 1$ and $\nu^+(-K) \le \nu+(9_{42}) + \nu^+( -Wh(T_{2,2a+1}) ) = 0$, so $\theta(K)$ also gives a better genus bound than $\nu^+(K)$ and $\nu^+(-K)$.
\end{example}

\begin{example}
Consider a connected sum such as 
\[
K = -Wh(T_{2,2a+1}) \# Wh(T_{2,2b+1}) \# Wh(T_{2,2c+1})
\]
with $a > b+c$. Then $\tau(K) = 1$, $\sigma(K) = 0$ and $\delta(K) = 4(a-b-c) > 0$, hence $\theta(-K) \ge 1$. Also $\theta(-K) \le \theta( Wh(T_{2,2a+1}) ) + \theta( -Wh(T_{2,2b+1})) + \theta(-Wh(T_{2,2c+1})) = 1$, hence $\theta(-K) = 1$. Since $\tau(K) > 0$, we see that $K$ is not smoothly $H$-slice in any negative definite $4$-manifold (with zero integral first homology) and since $\theta(-K) > 0$, we see that $K$ is not smoothly $H$-slice in any positive definite $4$-manifold either. Moreover, $K$ is topologically slice since it is a connected sum of Whitehead doubles.
\end{example}

\begin{example}
Most of the examples we have seen have either $\theta(K) = 0$ or $\theta(-K) = 0$. However it is certainly possible to have $\theta(K)$ and $\theta(-K)$ both non-zero. Consider for example $K = T_{2,5} \# (-Wh(T_{2,3}))$. Then $T_{2,5}$ is concordant to $K \# Wh(T_{2,3})$, hence
\[
2 = \theta(T_{2,5}) \le \theta(K) + \theta(Wh(T_{2,3})) = \theta(K) + 1
\]
which gives $\theta(K) \ge 1$. On the other hand $\sigma(K) = -4$ and $\delta(K) = 6 > -\sigma(K)/2$, so $\theta(-K) \ge 1$. Furthermore, $Wh(T_{2,3})$ can be unknotted by changing a positive crossing in the clasp to a negative crossing. So $T_{2,3}$ is obtained from $-K$ by changing a positive crossing to negative. Hence
\[
0 \le \theta( -K ) - \theta(T_{2,5}) \le 1,
\]
hence $\theta(-K) \ge 2$. Further, as $u(T_{2,5}) = 2$ and $u( Wh(T_{2,3})) = 1$, we get $\theta(K) + \theta(-K) \le u(K) \le 3$. So $\theta(K) = 1$, $\theta(-K) = 2$ and $u(K) = 3$.
\end{example}

\section{Genus bounds for definite $4$-manifolds}\label{sec:definite}

Let $X$ be a compact, oriented, smooth $4$-manifold with boundary $\partial X = S^3$. Note that $H^2(X , \partial X ; \mathbb{Z})$ and $H^2(X ; \mathbb{Z})$ are canonically isomorphic.

Let $K \subset S^3$ be a knot in $S^3$. Consider a smooth, oriented, properly embedded surface $\Sigma \to X$ bounding $K$. The surface $\Sigma$ has a relative homology class in $H_2(X , \partial X ; \mathbb{Z}) \cong H_2(X ; \mathbb{Z})$. Given a knot $K$ and a homology class $a \in H_2(X ; \mathbb{Z})$, we say that an embedded surface $\Sigma$ of the above type {\em represents} the pair $(K,a)$ if $\partial \Sigma = K$ and $[\Sigma]=a$. Define the {\em slice genus of $K$ with respect to $(X,a)$}, $g_4(K,X,a)$ to be the minimum genus of an embedded surface in $X$ which represents the pair $(K,a)$. Define the {\em $H$-slice genus of $K$ in $X$}, $g_4(K,X)$ to be the minimum genus of a smooth, oriented, properly embedded surface $\Sigma$ in $X$ bounding $K$ and such that $\Sigma$ is homologically trivial. Clearly $g_4(K,X,a)$ and $g_H(K,X)$ depend only on the isotopy class of $K$.

Assume that $H_1(X ; \mathbb{Z}) = 0$. Let $\Sigma \subseteq X$ be a properly embedded surface in $X$ bounding a knot $K$ and suppose further that the homology class $a = [\Sigma] \in H_2(X , \partial X ; \mathbb{Z})$ is divisible by a prime number $q$. Then we may construct the degree $q$ cyclic branched cover $\pi : W \to X$ branched over $\Sigma$.

Assume now that $X$ is negative definite, that is, the intersection form on $H^2(X ; \mathbb{Z})$ is negative definite. Then from Lemma \ref{lem:branch}, we have:

\begin{align*}
b_2(W) &=  q b_2(X) + (q-1)(2g), \\
\sigma(W) &= -q b_2(X) - \frac{(q^2-1)}{3q} [\Sigma]^2 + \sigma^{(q)}(K), \\
b_+(W) &= (q-1)g - \frac{(q^2-1)}{6q}[\Sigma]^2 + \frac{\sigma^{(q)}(K)}{2} , \\
\end{align*}
where $g$ is the genus of $\Sigma$.

\begin{theorem}\label{thm:in1}
Let $q$ be an odd prime and let $a \in H_2(X ; \mathbb{Z})$ be divisible by $q$. Then
\[
\delta_{j}^{(q)}(-K) \le -\frac{(q^2-1)}{6q} a^2 + \frac{\sigma^{(q)}(K)}{2}
\]
where
\[
2j = (q-1)g_4(K,X,a) - \frac{(q^2-1)}{6q}a^2 + \frac{\sigma^{(q)}(K)}{2}.
\]
Thus, setting $m =  -\frac{(q^2-1)}{6q}a^2$ we have that
\[
\frac{2 j^{(q)}(-K,m)}{q-1} \le g_4(K,X,a) - \frac{q+1}{6q}a^2 + \frac{\sigma^{(q)}(K)}{2(q-1)}
\]
and hence
\[
g_4(K,X,a) \ge \theta^{(q)}\left(K, m \right) + \frac{(q+1)}{6q}a^2.
\]
\end{theorem}
\begin{proof}
Choose an embedded surface $\Sigma$ of genus $g = g_4(K,X,a)$ representing $(K,a)$. Consider the cyclic $q$-fold branched cover $\pi : W \to X$ with branch locus $\Sigma$. This carries and action of $G = \mathbb{Z}_p$ generated by the covering transformation $\sigma$. The boundary of $W$ is $Y = \Sigma_q(K)$. Let $\mathfrak{s}$ be an invariant spin$^c$-structure on $W$. By Proposition \ref{prop:spinc2} we can choose $\mathfrak{s}$ so that $\widetilde{c} = c_1(\mathfrak{s}) = \pi^*(c)$ where $c$ is a characteristic on $X$ and such that $\mathfrak{s}|_{Y} = \mathfrak{s}_0$ is the distinguished spin$^c$-structure on $Y$. Since $X$ is negative definite, it has intersection form $diag(-1,-1, \dots , -1)$ (cap the boundary of $X$ off with a disc and apply Donaldson's theorem). Choose $c = (1,1, \dots ,1)$, so that $c^2 = -b_2(X)$.

Consider $W$ as a cobordism with ingoing boundary $\overline{Y} = \Sigma_q(\overline{K})$ and with empty outgoing boundary. We have that $H^+(W)^{G} \cong H^+(X) = 0$, since $X$ is negative definite. Thus the equivariant Froyshov inequality \cite[Theorem 5.3]{bh} gives:
\[
\delta( W , \mathfrak{s} ) + \delta_{\mathbb{Z}_q , S^j}(\overline{Y} , \mathfrak{s}_0) \le 0
\]
where
\[
2j = b_+(W) = (q-1)g_4(K,X,a) - \frac{(q^2-1)}{6q}[\Sigma]^2 + \frac{\sigma^{(q)}(K)}{2}.
\]
Re-arranging and using $\delta_j^{(q)}(-K) = 4\delta_{\mathbb{Z}_q , S^j}(\overline{Y} , \mathfrak{s}_0  )$, we have 
\begin{align*}
\delta_j^{(q)}(-K) \le -4\delta(W , \mathfrak{s})  &= \frac{ \sigma(W) - c_1(\mathfrak{s})^2}{2} \\
&= \frac{q \sigma(X)}{2} - \frac{(q^2-1)}{6q} [\Sigma]^2 + \frac{\sigma^{(q)}(K)}{2} + \frac{q b_2(X)}{2} \\
&=  - \frac{(q^2-1)}{6q} [\Sigma]^2 + \frac{\sigma^{(q)}(K)}{2}.
\end{align*}
This gives the result.
\end{proof}

Setting $a=0$, we have:
\begin{corollary}\label{cor:hodd}
For every odd prime $q$, we have
\[
g_H(K,X) \ge \theta^{(q)}(K).
\]
\end{corollary}

Now consider the case $q=2$. For any $x$, let $\eta(x) = \min_c \{ -(x+c)^2-b_2(X)\}$ be the minimum value of $-(x+c)^2-b_2(X)$ amongst all characteristics of $H^2(X ; \mathbb{Z})$. Since $X$ is negative definite, it has intersection form $diag(-1,-1, \dots , -1)$. Choosing a characteristic such that each entry of $x+c$ is $0$ or $1$, we see that $-b_2(X) \le \eta(x) \le 0$. Then we have

\begin{theorem}\label{thm:in2}
Let $a \in H^2(X ; \mathbb{Z})$ be divisible by $2$. Then for every characteristic $c \in H^2(X ; \mathbb{Z})$, we have
\[
\delta_{j}^{(2)}(-K) \le -\frac{a^2}{4} + \eta(a/2) + \frac{\sigma(K)}{2},
\]
where
\[
j = g_4(K,X,a) - \frac{a^2}{4} + \frac{\sigma(K)}{2}.
\]
Thus
\[
g_4(K,X,a) \ge j^{(2)}\left(K, m \right) + \frac{a^2}{4} - \frac{\sigma(K)}{2}
\]
where $m = -\frac{a^2}{4} + \eta(a/2)$. Equivalently,
\[
g_4(K,X,a) \ge \theta(K , m ) + \frac{1}{4}a^2.
\]
\end{theorem}
\begin{proof}
Set $x = a/2$. The proof is similar to the case that $q$ is an odd prime. The main difference is that we now have $\widetilde{c} = \pi^*( c + x)$, where $c$ is a characteristic on $X$. The Froyshov inequality gives
\[
2 \delta_{j}^{(2)}(-K) \le \sigma(W) - 2(x+c)^2 = -2(x+c)^2 - x^2 - 2b_+(X) + \sigma(K)
\]
where 
\[
j = b_+(W) = g_4(K,X,a) - x^2 + \sigma(K)/2.
\]
Minimising over $c$, we get
\[
\delta_{j}^{(2)}(-K) \le -x^2 + \eta(x) -\sigma(K)/2.
\]
This gives the result.
\end{proof}

If we set $a=0$, then $\eta(x) = \eta(0) = 0$, so we get
\begin{corollary}\label{cor:h2}
\[
g_H(K,X) \ge \theta(K).
\]
\end{corollary}

\begin{corollary}
Let $q$ be a prime and let $a \in H_2(X ; \mathbb{Z})$ be divisible by $q$. Then
\[
g_4(K,X,a) \ge \frac{\ell^{(q)}(-K)}{(q-1)} -\frac{3 \sigma^{(q)}(K)}{4(q-1)} + \frac{(q+1)}{4q}a^2.
\]
Moreover, if $q = 2$ the inequality can be improved to
\[
g_4(K,X,a) \ge \ell^{(2)}(-K) - \frac{3 \sigma(K)}{4} + \frac{3}{8}a^2 - \frac{1}{2}\eta( a/2 ).
\]
\end{corollary}
\begin{proof}
This follows from Theorems \ref{thm:in1}, \ref{thm:in2} together with the estimate for $\theta^{(q)}(K,m)$ given by Proposition \ref{prop:est}.
\end{proof}

\begin{example}
We compare our slice genus bound for surfaces bounding a $T_{3,6n+1}$-torus knot in a negative definite $4$-manifold with other bounds arising from the $\tau$-invariant and the signature.

Let $X$ be a compact, oriented, smooth negative definite $4$-manifold with $H_1(X ; \mathbb{Z}) = 0$ and with boundary $S^3$. Let $\Sigma \subset X$ be a smooth, oriented, properly embedded surface in $X$ bounding $K = T_{3,6n+1}$ and representing a class $a \in H_2(X , \partial X ; \mathbb{Z}) \cong H^2(X ; \mathbb{Z})$. We assume also that $a$ is divisible by $2$, so $a = 2x$ for some $x \in H^2(X ; \mathbb{Z})$. Choose a basis $e_1, \dots , e_r$ for $H^2(X ; \mathbb{Z})$ in which the intersection form is given by $diag(-1, \dots , -1)$. If $x = x_1 e_1 + \cdots + x_r e_r$ then $x^2 = -(x_1^2 + \cdots + x_r^2)$ and $\eta(x)$ is minus the number of $i$ such that $x_i$ is odd. It follows that $-x^2 + \eta(x)$ is divisible by $4$ and so
\[
\theta(T_{3,6n+1} , -x^2+\eta(x) ) = \max\{ 4n , 6n + x^2/2 - \eta(x)/2 \}.
\]
Hence Theorem \ref{thm:in2} gives the genus bound
\[
g(\Sigma) \ge \max\{ 4n + x^2 , 6n + 3x^2/2 - \eta(x)/2 \}.
\]
In particular, this gives
\begin{equation}\label{equ:theta}
g(\Sigma) \ge 6n + \frac{3x^2}{2} = 6n - \frac{3}{2}( x_1^2 + \cdots + x_r^2 ) + \frac{1}{2} \#\{ i \; | \; x_i \text{ odd} \}.
\end{equation}
Now we compare this with the genus bound in terms of the $\tau$ invariant \cite{os1}. We get $g(\Sigma) \ge a^2/2 + ||a||/2 + \tau(K)$, where $|| a_1 e_1 + \cdots + a_r e_r || = |a_1| + \cdots + |a_r|$. Since $\tau(T_{3,6n+1}) = 6n$, this simplifies to
\begin{equation}\label{equ:tau}
g(\Sigma) \ge 6n + 2x^2 + ||x|| = 6n - 2( x_1^2 + \cdots + x_r ) + ( |x_1| + \cdots + |x_r|).
\end{equation}
One also gets two genus bounds in terms of the signature \cite{gi,vi}:
\begin{equation}\label{equ:sig1}
g(\Sigma) \ge 4n+x^2 = 4n - ( x_1^2 + \cdots + x_r^2)
\end{equation}
and
\begin{equation}\label{equ:sig2}
g(\Sigma) \ge -4n-x^2 - r = -4n - r + (x_1^2 + \cdots + x_r^2).
\end{equation}
We have obtained four genus bounds. Our bound using the $\theta$ invariant is (\ref{equ:theta}). The Ozsv\'ath--Szab\'o bound using the $\tau$ invariant is (\ref{equ:tau}) and the two bounds (\ref{equ:sig1}) and (\ref{equ:sig2}) which use the signature. It is easily seen that (\ref{equ:theta}) is always at least as good as (\ref{equ:tau}) and is often an improvement. It is also easily seen that when $| x^2 |$ is relatively small, say $| x^2 | \le 4n$, our genus bound (\ref{equ:theta}) is a better genus bound than (\ref{equ:sig1}) or (\ref{equ:sig2}). On the other hand, when $|x^2|$ sufficiently large, the bound $(\ref{equ:sig2})$ is much better than all the other bounds. In fact the right hand sides of the bounds (\ref{equ:theta}), (\ref{equ:tau}) and (\ref{equ:sig1}) all become negative when $|x^2|$ is sufficiently large.

\end{example}


\bibliographystyle{amsplain}

\end{document}